\documentclass{article}
\usepackage[utf8]{inputenc}
\usepackage[english]{babel}
\usepackage[margin=1in]{geometry}
\usepackage{blindtext}
\usepackage{amssymb,amsmath}
\usepackage{graphicx}
\usepackage{caption}
\usepackage{subcaption}
\usepackage{mathtools, nccmath}
\usepackage{xcolor}
\usepackage{mathrsfs}
\usepackage{amsthm}
\newtheorem{theorem}{Theorem}[section]
\newtheorem{corollary}{Corollary}[theorem]
\newtheorem{lemma}[theorem]{Lemma}
\theoremstyle{remark}
\theoremstyle{definition}

\usepackage{enumitem}   
\newcommand{\overbar}[1]{\mkern 1.5mu\overline{\mkern-1.5mu#1\mkern-1.5mu}\mkern 1.5mu}
\newcommand{\E}{\mathbf{E}}

\newcommand{\tr}{\mathrm{tr}}

\numberwithin{equation}{section}

\title{Spiked Singular Values and Vectors \\ Under Extreme Aspect Ratios}
\author{Michael J. Feldman}
\begin{document}
\date{}
\maketitle

\begin{abstract}
The behavior of the leading singular values and vectors of noisy low-rank matrices is fundamental to many statistical and scientific problems. Theoretical understanding currently derives from asymptotic analysis under one of two regimes: (1) the {\it classical} regime, with a fixed number of rows and large number of columns, or vice versa, and (2) the {\it proportional} regime, with large numbers of rows and columns, proportional to one another. This paper is concerned with the {\it disproportional} regime, where the matrix is either ``tall and narrow'' or ``short and wide'': we study sequences of matrices of size $n \times m_n$ with aspect ratio $ n/m_n \rightarrow 0$ or $n/m_n \rightarrow \infty$ as $n \rightarrow \infty$. This regime has important ``big data'' applications.

Theory derived here shows that the displacement of the empirical singular values and vectors from their noise-free counterparts and the associated phase transitions---well-known under proportional growth asymptotics---still occur in the disproportionate setting. They must be quantified, however, on a novel scale of measurement that adjusts with the changing aspect ratio as the matrix size increases. 
In this setting, the top singular vectors corresponding to the longer of the two matrix dimensions are asymptotically uncorrelated with the noise-free signal.
\end{abstract}

\section{Introduction \label{sec:1}}

 The low-rank signal-plus-noise model is a simple statistical model of data with latent low-rank structure. Data $\widetilde X_n$ is the sum of a low-rank matrix and white noise:
\begin{align} \phantom{\, .} \mfrac{1}{\sqrt{m_n}} \widetilde{X}_n = \sum_{i=1}^r \theta_i u_i v_i^\top + \mfrac{1}{\sqrt{m_n}} X_n  \, , \label{model} \end{align}
where $\theta_i \in \mathbb{R}$ are signal strengths, $u_i \in \mathbb{R}^n$ and $v_i \in \mathbb{R}^{m_n}$ are the left and right signal vectors, and the noise matrix $X_n \in \mathbb{R}^{n \times m_n}$ contains independent and identically distributed (i.i.d.) entries with mean zero and variance one. The noise matrix is normalized so that rows are asymptotically unit norm. 

\subsection{Proportionate growth asymptotic}
Recent work studies this model in the high-dimensional setting where $n$ and $m_n$ are large, in particular, where $n$ and $m_n$ are of comparable magnitude. Such analyses derive limiting behavior of sequences $\widetilde X_n$ as 
\[ \phantom{\, ,} n, m_n \rightarrow \infty \, , \,\,\,\, \frac{n}{m_n} \rightarrow \beta > 0 \, .
\]
Here, the parameter $\beta > 0$ is the limiting aspect ratio of the data; thus for $\beta = 1$, the matrices are effectively square for all large $n$.  Baik, Ben Arous, and P\'ech\'e \cite{BBP}, Baik and Silverstein \cite{BkS04}, and Paul \cite{P07} study eigenvalues of the spiked covariance model, which is closely related to model (\ref{model}). 
Benaych-Georges and Nadakuditi \cite{BGN12} derive asymptotic properties of the  singular value decomposition of $\widetilde X_n$ (details of model (\ref{model}) are provided in Section \ref{sec1.3.0}). Two phenomena arise not present in classical fixed-$n$ asymptotics:

{\it Leading eigenvalue displacement}. Let $\tilde \lambda_1 \geq  \cdots \geq \tilde \lambda_n $ denote the eigenvalues of $\widetilde S_n = \frac{1}{m_n} \widetilde X_n \widetilde X_n^\top$, the sample covariance matrix. We assume signal strengths $\theta_1 > \cdots > \theta_r$ are constant and distinct. The leading eigenvalues of $\widetilde S_n$ are inconsistent estimators of the eigenvalues of $\E \widetilde S_n$: for fixed $i \geq 1$,  
\begin{align}
\phantom{\,,} \tilde \lambda_i & \xrightarrow{a.s.}  \begin{dcases} \frac{(1+ \theta_i^2)(\beta + \theta_i^2)}{\theta_i^2}  \,, &  i \leq r, \, \theta_i > \beta^{1/4} \\   (1 + \sqrt{\beta})^2 \,, & \text{otherwise} \end{dcases} \, .  \nonumber \end{align} 
The singular values of $\frac{1}{\sqrt{m_n}} \widetilde X_n$ are of course determined by the eigenvalues of $\widetilde S_n$, and vice versa. We will occasionally switch between the two without comment.

{\it Leading singular vector inconsistency}. Let  $\tilde u_1, \ldots, \tilde u_n$ and $\tilde v_1, \ldots, \tilde v_m$ denote the left and right singular vectors of  $ \widetilde X_n$. The leading singular vectors $\tilde u_1, \ldots, \tilde u_r$ and $\tilde v_1, \ldots, \tilde v_r$ are inconsistent estimators of the left and right signal vectors. For $1 \leq i \leq n$ and $1 \leq j \leq r$,

\begin{align*} & \hspace{3cm} | \langle \tilde u_i,  u_j \rangle |^2 \xrightarrow{a.s.}   1 - \frac{\beta (1+ \theta_i^2)}{\theta_i^2(\theta_i^2 + \beta)}  \,, \quad | \langle \tilde v_i, v_j \rangle |^2\xrightarrow{a.s.}1 - \frac{\beta+ \theta_i^2}{\theta_i^2(\theta_i^2 + 1)} \,,  && i = j, \, \theta_i > \beta^{1/4}  \, , \nonumber
\end{align*}
and $| \langle \tilde u_i,  u_j \rangle |, | \langle \tilde v_i,  v_j \rangle | \xrightarrow{a.s.} 0$ otherwise.

\subsection{Disproportionate growth asymptotic}

Contrary to the proportionate growth asymptotic, the dimensions of many large data matrices are not comparable. For example, Novembre et al.\ \cite{N08} demonstrate a genetic dataset with 3,000 rows (people) and 250,000 columns (genetic measurements) has low-rank structure. Here the aspect ratio is $\beta \approx 3/250$. 

This paper considers the signal-plus-noise model under disproportional growth:
 \[ \phantom{\,. }n, m_n \rightarrow \infty \, , \,\,\,\,  \beta_n = n/m_n \rightarrow 0 \, . \] Transposing $\widetilde X_n$ if necessary, our results also apply to $\beta_n \rightarrow \infty$.
Substitution of $\beta = 0$ into the above proportional-limit formulas heuristically suggests $\lambda_i \xrightarrow{a.s.} 1 + \theta_i^2$ and $ | \langle  \tilde  u_i, u_j \rangle |^2  \xrightarrow{a.s.}  1$ (right signal vectors are partially recovered). Indeed, these formulas are corollaries of Theorems 2.9 and 2.10 of \cite{BGN12}. Thus, under $\beta_n \rightarrow 0$ and constant signal strengths, leading eigenvalues of $\widetilde S_n$  consistently estimate those of $\E \widetilde S_n$, and left singular vectors $\tilde u_i$, corresponding to the shorter matrix dimension, consistently estimate left signal vectors $u_i$. In particular, leading eigenvalue displacement and (left) singular vector inconsistency no longer occur, and  seemingly no  phase transition exists.

Our main contribution is the discovery of a vanishingly small phase transition located at $\beta_n^{1/4}$. Signal strengths $\theta_i = \theta_{i,n}$, hitherto fixed, now vanish. For signal strengths  above this ``microscopic'' threshold, left singular vectors partially recover their signal counterparts---enabling signal estimation at signal-to-noise ratios previously thought insufficiently strong. Right singular vectors, corresponding to the longer matrix dimension, are asymptotically uncorrelated with the signal.   

We introduce a particular calibration ${\theta_n = \tau \cdot \beta_n^{1/4}}$ between signal strength $\theta_n$ and $\beta_n$ depending on a new (constant) formal parameter $\tau$. Under this calibration, we rigorously establish a new set of formulas for the limiting displacement of singular values and inconsistency of singular vectors, in explicit terms of the parameter $\tau$. These results may be heuristically derived by substitution of $\theta_n = \tau \cdot \beta_n^{1/4}$ into proportional formulas and taking limits. Many important applications previously made under proportional growth may in parallel fashion now be rigorously made under disproportional growth. For example, in the high-dimensional Gaussian mixture model with many more parameters $m$ than samples $n$, our theory gives a phase transition for recovery of length-$n$ signal vectors (encoding cluster membership). The optimal singular value hard-thresholding level for low-rank matrix recovery  as $\beta_n \rightarrow 0$ may be calculated, analogous to proportional results of Donoho and Gavish \cite{GD14}. Furthermore, there are new potential uses. Spiked tensors closely relate to the disproportionate asymptotic:  Montanari and Richard propose a tensor unfolding algorithm \cite{MR14} that produces noisy low-rank matrices of size $n \times n^k$, $k \geq 2$ (equivalently, $\beta_n = n^{-(k-1)})$. Precise asymptotic analysis of tensor unfolding is made possible by our results. 

The related spiked covariance model is studied in the disproportionate asymptotic by Bloemendal, Knowles, Yau, and Yin \cite{BKYY16}. An advantage of this work is that $\beta_n$ is permitted to vanish or diverge at any rate, while \cite{BKYY16} requires the existence of a constant $k > 0$ such that $n^{1/k} \leq m_n \leq n^k$. 

\subsection{Assumptions and notation}  \label{sec1.3.0}
We make the following assumptions:
\begin{itemize}
\item[] (A1.\ Noise) The entries of $X_n = (x_{ij}: 1 \leq i \leq n, 1 \leq j \leq m_n)$ are i.i.d.\ with $\E x_{11} = 0$, $\E x_{11}^2 = 1$, and $\E x_{11}^4 < \infty$.
    
\item[] (A2.\ Signal vectors) The signal rank $r$ is fixed. Let $u_i$ and $v_i$ be the $i$-th columns of $U_n \in \mathbb{R}^{n\times r}$ and $V_n \in \mathbb{R}^{m_n \times r}$, respectively. The entries of $\sqrt{n} U_n$ and $\sqrt{m_n}V_n$ are i.i.d.\ with mean zero, variance one, and finite eighth moment.

\item[] (A3.\ Signal Strength) The signal strengths $\theta_i = \theta_{i,n}$ obey
    \[  \theta_i = \tau_i \beta_n^{1/4}(1+\varepsilon_{i,n}) 
    \] for distinct constants $\tau_1 > \ldots > \tau_r \geq 0$ and (nonrandom) sequences $\varepsilon_{i,n} \rightarrow 0$. 
\end{itemize}
Alternatively, A1 and A2 may be replaced with
\begin{itemize}
    \item[] (B1.\ Noise) The entries of $X_n = (x_{ij}: 1 \leq i \leq n, 1 \leq j \leq m_n)$ are independent standard Gaussians.
    \item[] (B2.\ Signal vectors) The signal rank $r$ is fixed.  Let $u_i$ and $v_i$ be the $i$-th columns of $U_n \in \mathbb{R}^{n\times r}$ and $V_n \in \mathbb{R}^{m_n \times r}$, respectively. $U_n$ and $V_n$ have orthonormal columns: $U_n^\top U_n = V_n^\top V_n = I_r$. 
\end{itemize}

Henceforth, we supress the subscript $n$ of $m_n$.  Unless explicitly stated otherwise, all results and lemmas in this paper hold simultaneously under (A1, A2, A3) and (B1, B2, A3). These assumptions are similar to those of \cite{BGN12}. The primary difference is \cite{BGN12} permits anisotropic noise $X_n$, provided two key conditions are satisfied: (1) the empirical spectral distribution (ESD) of $S_n = m^{-1} X_n X_n^\top$ converges almost surely weakly to a deterministic, compactly supported measure $\mu$ and (2) the maximum eigenvalue of $S_n$ converges almost surely to the supremum of the support of $\mu$. This paper considers  isotropic noise for the following reasons: firstly, the extension from proportionate growth to  disproportionate growth and vanishing signals is most clear in the fundamental, isotropic case. Secondly, the anisotropic extension requires disproportionate analogs of (1) and (2), such as (3) convergence of the ESD of $\beta_n^{-1/2}(m^{-1}\Sigma_n^{1/2} X_n X_n^\top \Sigma_n^{1/2} - \Sigma_n)$ to a deterministic, compactly supported measure $\mu$ and (4) convergence of the maximum eigenvalue of this matrix to the supremum of the support of $\mu$. While (3) is established by Wang and Paul \cite{LW}, (4) is proven only in the isotropic case, by Chen and Pan \cite{CP12} (see Lemma \ref{lem:max_eig}).  Assumption A2 is a relaxation of assumption 2.4 of \cite{BGN12}, which requires sub-Gaussian moments of the entries of $\sqrt{n}U_n$ and $\sqrt{m_n}V_n$.  The assumption of distinct signal strengths (A3) is for simplicity.

In this paper, ``almost surely eventually'' means with probability one, a sequence of events indexed by $n$ occurs for all sufficiently large $n$. The notation      
 $a_n \lesssim b_n$ means eventually $a_n \leq C b_n$ for some universal constant $C$; $a_n \asymp b_n$ means $b_n \lesssim a_n \lesssim b_n$; and $a_n = O(b_n)$ means $|a_n| \lesssim b_n$. For a parameter $\ell$, $a_n \lesssim_\ell b_n$ means  $a_n \leq C(\ell) b_n$.

\subsection{Results} \label{sec1.3}

As stated above, we study vanishing signal strengths $\theta_i = \tau_i \beta_n^{1/4} (1 + \varepsilon_{i,n})$. The formal parameter $\tau$ describes the signal strength on a refined scale of analysis, with a phase transition occurring at $\tau = 1$. While for $\tau < 1$, both left and right singular vectors are asymptotically uncorrelated with the underlying signal vectors, for $\tau > 1$, left singular vectors correlate with left signal vectors. Define the following useful index:  
\[ \phantom{\,.} i_0= \max \{1 \leq i \leq r: \, \tau_i > 1 \} \, . \]

\begin{theorem} \label{thrm:eig_conv} Let $\widetilde X_n$ denote a sequence of signal-plus-noise models satisfying the above assumptions. As $n \rightarrow \infty$ and $\beta_n \rightarrow 0$, for any fixed $i \geq 1$, 
\begin{align}
&  \hspace{2cm} \frac{ \tilde \lambda_i - 1}{\sqrt{\beta_n}}  \xrightarrow{a.s.} \begin{dcases}  \tau_i^2 + \frac{1}{\tau_i^2}   &  i \leq i_0  \\
2 & i > i_0 \end{dcases} \label{1.2}  \, .
\end{align} 
For $1 \leq i, j \leq i_0$, 
\begin{align} & | \langle \tilde u_i, u_j \rangle |^2 \xrightarrow{a.s.}  \delta_{ij} \cdot (1 - \tau_i^{-4})  \,, && | \langle \tilde v_i, v_j \rangle |^2 \xrightarrow{a.s.} 0 \,. \label{qwerty} \end{align}  
\end{theorem}

Theorem \ref{thrm:eig_conv} is a consequence of the following stronger result: for $i \leq i_0$, define
\begin{align} \phantom{\,.} \bar \lambda_i = \frac{(1+\theta_i^2)(\beta_n + \theta_i^2)}{\theta_i^2} = 1 + (\tau_i^2 + \tau_i^{-2}) \sqrt{\beta_n} + \beta_n  \, . 
\end{align}

\begin{theorem} \label{thrm:cosine}
Adopt the setting of Theorem \ref{thrm:eig_conv}. For $1 \leq i, j \leq i_0$, $i \neq j$, and $\ell < 1/4$, almost surely,
\begin{align}   \phantom{\, .} |\tilde  \lambda_i - \bar \lambda_i| \lesssim (n^{-\ell}  + |\varepsilon_{i,n}| ) \sqrt{\beta_n} \, , \label{1.5}
\end{align}
\vspace{-.2in}
\begin{align}
   & | \langle \tilde u_i, u_i \rangle |^2 = 1 - \tau_i^{-4} + O(n^{-\ell/2} + |\varepsilon_{i,n}| + \sqrt{\beta_n}) \,, && | \langle \tilde u_i, u_j \rangle |^2 \lesssim n^{-2\ell} \,, \label{1.8} \\
   &| \langle \tilde v_i, v_i \rangle |^2 \lesssim \sqrt{\beta_n} \,, && |\langle \tilde v_i, v_j \rangle |^2 \lesssim n^{-2\ell} \sqrt{\beta_n} \,. \label{1.9}
\end{align}
\end{theorem}



\section{Preliminaries} 

\subsection{Overview} \label{005}

This section reviews the framework developed in Benaych-Georges and Nadakuditi's work \cite{BGN11, BGN12} addressing the $\beta_n \rightarrow \beta > 0 $, fixed signal setting, followed by a discussion of adaptations required to study the $\beta_n \rightarrow 0$, vanishing signal setting. 

As before, let $S_n = \frac{1}{m} X_n X_n^\top$ be the sample covariance matrix of noise. Let $F_n$  denote the ESD of $S_n$ and $\overbar F_n$ the Marchenko-Pastur law with parameter $\beta_n$. $s_n(z)$ and $\bar s_n(z)$ will denote the corresponding Stieltjes transforms, defined as follows:
 \begin{align} & s_n(z) = \int \frac{1}{\lambda - z} dF_n(\lambda) = \frac{1}{n} \tr(S_n - zI_n)^{-1} \,, \\ &
\bar s_n(z) = \int \frac{1}{\lambda - z} d \overbar F_n(\lambda) =  \frac{1 - \beta_n - z + \sqrt{(1 + \beta_n - z)^2 - 4 \beta_n }}{2\beta_n z} \,, \label{81237}
\end{align}
defined for $z$ outside the support of $F_n, \overbar F_n$. The square root is the principal branch.

 Let $\Theta_n = \text{diag}(\theta_1, \ldots, \theta_r)$. Recall that $U_n$ and $V_n$, defined in Section \ref{sec1.3.0}, contain as columns the left and right signal vectors. The following $2r \times 2r$ matrices will be the central objects of study, providing insight into the spectral decomposition of $\widetilde S_n$:
\begin{align} \widetilde{M}_n(z) & = \begin{bmatrix}  \sqrt{z} U_n^\top (S_n - z I_n)^{-1} U_n &  \frac{1}{\sqrt{m}}U_n^\top (S_n - z I_n)^{-1} X_n V_n + \Theta_n^{-1} \\  \frac{1}{\sqrt{m}} V_n^\top X_n^\top (S_n - z I_n)^{-1} U_n + \Theta_n^{-1} &  \sqrt{z} V_n^\top (\frac{1}{m} X_n^\top X_n - z I_m)^{-1} V_n \end{bmatrix} \,, \label{1time} \\
M_n(z) &  = \begin{bmatrix}  \sqrt{z} s_n(z) I_r & \Theta_n^{-1} \\ \Theta_n^{-1} & \big( \beta_n \sqrt{z} s_n(z) - \frac{1-\beta_n}{\sqrt{z}} \big) I_r \end{bmatrix}   \,, \label{1time2} \\
 \overbar{M}_n(z) & = \begin{bmatrix}  \sqrt{z} \bar s_n(z) I_r & \Theta_n^{-1} \\ \Theta_n^{-1} & \big( \beta_n \sqrt{z} \bar s_n(z) - \frac{1-\beta_n}{\sqrt{z}} \big) I_r \end{bmatrix}  \,. 
\end{align}
Pivotal to their utility is this essential observation:
\begin{lemma} \label{lem:4.1BGN} (Lemma 4.1 of \cite{BGN12}) 
The eigenvalues of $\widetilde{S}_n$ that are not eigenvalues of $S_n$ are the real values $z$ such $\widetilde{M}_n(z)$ is non-invertible. 
\end{lemma}
Lemma \ref{lem:4.1BGN} focuses the study of the leading eigenvalues of $\widetilde S_n$---those resulting from the signal component of $\widetilde X_n$---away from the {\it increasing}-dimensional sequence $\widetilde S_n$ to the {\it fixed}-dimensional sequence $\widetilde M_n$ and its associates, $M_n$ and $\overbar M_n$. The roots of $\det \widetilde M_n(z)$ will be shown to localize near the roots of $\det \overbar M_n(z)$. Using the closed form (\ref{81237}) of the Stieltjes transform $\bar s_n$, $\det \overbar M_n$ has explicitly known roots. As \cite{BGN12} shows, these roots are precisely $\bar \lambda_1, \ldots, \bar \lambda_{i_0}$, introduced in Section \ref{sec1.3} (assuming $\varepsilon_{i,n} = 0$):
\begin{lemma} (Section 3.1 of \cite{BGN12}) \label{lem:Dbar} Define
\begin{align} \phantom{\,.} 
 \overbar{D}_n(z) &=  \sqrt{z} \bar s_n(z) \Big( \beta_n \sqrt{z}  \bar s_n(z) - \frac{1-\beta_n}{\sqrt{z}}\Big)  = -\frac{1+\beta_n - z + \sqrt{(1+\beta_n-z)^2 - 4 \beta_n}}{2 \beta_n} \,. \label{696}  
\end{align}
 As the blocks of $\overbar M_n(z)$ commute,
\begin{align} \phantom{\,.}
 \det \overbar{M}_n(z) = \det(\overbar D_n(z) I_r - \Theta_n^{-2})
 = \prod_{i=1}^r ( \overbar{D}_n(z) - \theta_i^{-2}) \, . \label{7777}
\end{align}
On the real axis, $\overbar D_n(z)$ maps $ \big( (1+\sqrt{\beta_n})^2, \infty)$ in a one-to-one fashion onto $(0, \beta_n^{-1/2})$. The compositional inverse $\overbar D_n^{(-1)}(t)$ exists on the real interval $(0, \beta_n^{-1/2})$ and is given in closed form by 
 \begin{align}
 \phantom{\,. }  & \overbar D_n^{(-1)}(t)   = \frac{(t+1)(\beta_n t +1)}{t} \,, & 0 < t < \beta_n^{-1/2} \,.
 \end{align}
 Thus, assuming $\varepsilon_{i,n} = 0$, $\bar \lambda_i = \overbar D_n^{(-1)}(\theta_i^{-2})$ is a root of $\det \overbar M_n(z)$ provided that $\theta_i^{-2} < \beta_n^{-1/2}$, or $i \leq i_0$.
\end{lemma}

The singular vectors of $\widetilde X_n$ are related to $U_n$ and $V_n$ via the next lemma. 

\begin{lemma}\label{lem: BG5.1}(Lemma 5.1 of \cite{BGN12})
Let $\tilde \sigma  = \sqrt{ \tilde \lambda}$ be a singular value of $\frac{1}{\sqrt{m}} \widetilde X_n$ that is not a singular value of $\frac{1}{\sqrt{m}} X_n$ and let $\tilde u, \tilde v$ be the corresponding  singular vectors. Then the column vector \begin{align}  \begin{bmatrix} \Theta_n V_n^\top \tilde v \\ \Theta_n U_n^\top \tilde u \end{bmatrix}    \nonumber \end{align}
belongs to the kernel of $\widetilde M_n(\tilde \lambda)$. Moreover, for $P_n = \sum_{i=1}^r \theta_i u_i v_i^\top$,
\begin{align}
   \phantom{\,.} 1 & =    \tilde \lambda \tilde v^\top P_n^\top ( S_n - \tilde \lambda I_n)^{-2}  P_n \tilde v + \frac{1}{m}\tilde u^\top P_n X_n^\top (S_n - \tilde \lambda I_n)^{-2} X_n P_n^\top \tilde u \nonumber \\  & \,\,\,\,\,\, +   \frac{\tilde \sigma}{\sqrt{m}} \tilde u^\top P_n X_n^\top ( S_n - \tilde \lambda I_n)^{-2} P_n \tilde v +  \frac{\tilde \sigma}{\sqrt{m}} \tilde v^\top P_n^\top ( S_n - \tilde \lambda I_n)^{-2} X_n P_n^\top \tilde u  \label{765} \,.
\end{align}

\end{lemma}

The results quoted above lie at the heart of the strategy  of \cite{BGN11, BGN12} for proportionate growth and fixed signals. That strategy is able to employ convenient convergence arguments. Indeed, as $\beta_n \rightarrow \beta > 0$, almost surely, the ESD of noise eigenvalues $F_n$ converges weakly to $\overbar F_\beta$,  the Marchenko-Pastur law with parameter $\beta$. Moreover, the top noise eigenvalue $\lambda_1 \xrightarrow{a.s.} (1+\sqrt{\beta})^2$, the upper edge of the bulk (the support of $\overbar F_\beta$). Let $[a,b]$ be a compact interval outside the bulk: $(1+\sqrt{\beta})^2 < a$. By the Arzela-Ascoli theorem, the Stieltjes transform of the noise eigenvalues $s_n(z)$ converges uniformly on $[a,b]$ to $\bar s_\beta(z)$, the Stieltjes transform of $\overbar F_\beta$. The implications of this key observation include the uniform convergence of $M_n(z)$ to a non-random matrix $\overbar M_\beta(z)$ (the fixed-$\beta$ analog of $\overbar M_n(z)$). From the uniform convergence of $M_n(z)$, convergence of the roots of $\det M_n(z)$ to those of $\det \overbar M_\beta(z)$ is established.

Now, consider the disproportionate asymptotic. Leading eigenvalue displacement and singular vector inconsistency no longer occur under fixed signals. Rather, the regime of interest is one where  signal strengths vanish as $\beta_n^{1/4}$, mandating study of $M_n(z)$ and  $\overbar M_n(z)$ very near one. The convergences which were previously so helpful under $\beta_n \rightarrow \beta > 0$ and fixed signals are not useful as $\beta_n \rightarrow 0$. The limits of $F_n$ and $\overbar M_n(z)$ are degenerate, as all noise eigenvalues converge to one. The upper edge of the support of $F_n$ lies approximately at $(1+\sqrt{\beta_n})^2$. Theorem \ref{thrm:eig_conv} tells us the leading eigenvalues of $\widetilde S_n$ may still emerge from the bulk, but only by a multiple of $\sqrt{\beta_n}$. Phenomena of interest are no longer made visible by taking limits, but must be captured by detailed finite-$n$ analysis. This will require bounds on the convergence rate of $s_n(z)$ to $\bar s_n(z)$, developed in Section \ref{sec:b}, as well as careful tracking of the precise size of remainders, where previously it sufficed to know that such remainders tended to zero. The proofs of Theorems \ref{thrm:eig_conv} and \ref{thrm:cosine} are refinements of the proofs of Theorems 2.9 and 2.10 of \cite{BGN12} and Lemma 6.1 of \cite{BGN11}. 

\subsection{Preliminary lemmas} \label{sec:2.2}


This section contains lemmas used in the proof of Theorem \ref{thrm:cosine}. Let $\lambda_1 \geq \lambda_2 \geq \cdots \geq \lambda_n $ denote the eigenvalues of $S_n = \frac{1}{m} X_n X_n^\top$.

\begin{lemma} \label{asdf1}
    Let $Y_n$ be the matrix obtained by entrywise truncation and normalization of $X_n$ as in Lemma \ref{lem:max_eig2} and $\widetilde Y_n = \sqrt{m} \sum_{i=1}^r \theta_i u_i v_i^\top + Y_n$.     For $1 \leq i \leq n$, almost surely,
    \begin{align*} \phantom{\,.}
        | \tilde \lambda_i - \lambda_i(m^{-1} \widetilde Y_n \widetilde Y_n^\top) | \lesssim \frac{\sqrt{\beta_n}}{n} \, . 
    \end{align*}
Moreover, if $|\lambda_{i}(m^{-1} \widetilde Y_n \widetilde Y_n^\top) - \lambda_{i+1}(m^{-1} \widetilde Y_n \widetilde Y_n^\top)| \asymp \sqrt{\beta_n}$  for $1 \leq i, j \leq i_0$,
\begin{align*}
    & |\langle \tilde u_i, u_{j} \rangle |^2 =  |\langle \tilde u_{Y,i}, u_{j} \rangle | ^2 + O_{a.s.}(n^{-1}) \, , &|\langle \tilde v_i, v_{j} \rangle | ^2 =  |\langle \tilde v_{Y,i}, v_{j} \rangle |^2 + O_{a.s.}(n^{-1}) \, ,
\end{align*}
where $\tilde u_{Y,1}, \ldots \tilde u_{Y,i_0}$ and $\tilde v_{Y,1}, \ldots, \tilde v_{Y,i_0}$ are the leading left and right singular vectors of $\widetilde Y_n$, respectively. 
\end{lemma}
\begin{proof}
    This follows from Lemma \ref{lem:trunc}. 
\end{proof}
\noindent  By Lemma \ref{asdf1},  without loss of generality, we henceforth assume the entries of $X_n$ are truncated and normalized in accordance with  Lemma \ref{lem:max_eig2}. Comparing Lemma \ref{asdf1} and Theorem \ref{thrm:cosine}, the errors induced by truncation and normalization are negligible; it therefore suffices to prove Theorem \ref{thrm:cosine} in this modified setting. 
By Lemma \ref{lem:max_eig2}, $\bar \lambda_1, \ldots, \bar \lambda_{i_0}$ emerge from the ESD of $S_n$ with high probability, a fact we shall use repeatedly: for any $\eta, \ell > 0$, 
\begin{align} \phantom{\,.}\Pr(  \lambda_1 \geq 1 + (2+\eta/2) \sqrt{\beta_n}) =  o(n^{-\ell}) \, .  \label{111}
\end{align}





The following lemma bounds $s_n(z)$ and $\bar s_n(z)$ on a region containing $\bar \lambda_1, \ldots, \bar \lambda_{i_0}$.  The first point is a standard result. See, for example, Lemma $8.17$ of \cite{BS_SpAn}. The second point relies on (\ref{111}). 
\begin{lemma} \label{lem:s_n bound} Let $\eta > 0$ and $\mathcal{Z}_{\eta, n} = \{z: \Re(z)\geq 1 + (2+ \eta) \sqrt{\beta_n}\}$. Then,
\begin{align*}
\phantom{\,.} & \sup_{z \in \mathcal{Z}_{\eta, n}}\, |\bar s_n(z)| \lesssim \beta_n^{-1/2} \,, && \sup_{z \in \mathcal{Z}_{\eta, n}}\,  \Big| \frac{d}{dz} \bar s_n(z) \Big| \lesssim \beta_n^{-1}  \, .
\end{align*}
Furthermore, almost surely, 
\begin{align*}
\phantom{\,.} & \sup_{z \in \mathcal{Z}_{\eta, n}}\,  |s_n(z)| \lesssim \beta_n^{-1/2}\,, && \sup_{z \in \mathcal{Z}_{\eta, n}}\,  \Big| \frac{d}{dz} s_n(z) \Big| \lesssim  \beta_n^{-1}  \,.
\end{align*}
The implied coefficients depend on $\eta$ only.
\end{lemma}
\begin{proof}
As the upper edge of the support of the Marchenko-Pasteur law lies at $(1+\sqrt{\beta_n})^2$,
\begin{align}
&  \phantom{\,.}  \inf_{z \in \mathcal{Z}_{\eta, n}} \inf_{\lambda \in \text{supp} \overbar F_n} |\lambda - z| \gtrsim \sqrt{\beta_n} \,,  && \inf_{z \in \mathcal{Z}_{\eta, n}} \inf_{\lambda \in \text{supp} \overbar F_n} |\lambda - z|^2 \gtrsim  \beta_n \,. \label{1000}
\end{align}
The first claims of the lemma now follow from the integral representations
\begin{align*}
 &  \phantom{\,.} \bar s_n(z) = \int \frac{1}{\lambda - z} d \overbar F_n(\lambda) \,, && \frac{d}{dz} \bar s_n(z) = -\int \frac{1}{(\lambda - z)^2} d \overbar F_n(\lambda) \,.
\end{align*}
By (\ref{111}), almost surely eventually, 
\begin{align}
&  \phantom{\,.}  \inf_{z \in \mathcal{Z}_{\eta, n}} \min_{1 \leq \alpha \leq n} |\lambda_\alpha - z| \geq \frac{\eta \sqrt{\beta_n}}{2} \,,  && \inf_{z \in \mathcal{Z}_{\eta, n}} \min_{1 \leq \alpha \leq n} |\lambda_\alpha - z|^2 \geq \frac{\eta^2 \beta_n}{4} \,. \label{419}
\end{align}
Together with 
\begin{align*}
\phantom{\,.} & s_n(z) = \frac{1}{n} \sum_{\alpha=1}^n \frac{1}{\lambda_\alpha - z} \,, && \frac{d}{dz} s_n(z) = -\frac{1}{n} \sum_{\alpha=1}^n \frac{1}{(\lambda_\alpha - z)^2}  \,,
\end{align*}
we obtain the second group of claims. 
\end{proof}
The follow lemmas bound the deviation of $\widetilde M_n(z)$ (defined in (\ref{1time})) and its entrywise derivative $\frac{d}{dz} \widetilde M_n(z)$ from $M_n(z)$ (defined in (\ref{1time2})) and   $\frac{d}{dz} \widetilde M_n(z)$, respectively. Under assumption A2, $M_n(z)$ and $\frac{d}{dz} M_n(z)$ are precisely the  entrywise conditional expectations of $\widetilde M_n(z)$ and $\frac{d}{dz} \widetilde M_n(z)$ given $X_n$.
\begin{lemma} \label{lem123} Let $a_n \geq 1+(2+\eta)\sqrt{\beta_n}$ be a bounded sequence. For any  $\ell < 1/4$ and $1 \leq i,j \leq r$, almost surely, 
\begin{align}  &  \sup_{ |z- a_n| \leq n^{-1/4}\sqrt{\beta_n}} \,  \big | ( \widetilde M_n(z) - M_n(z)  )_{i,j} \big|  \lesssim  n^{-\ell} \beta_n^{-1/2}   \,,  \label{411}\\  &
  \sup_{ |z - a_n| \leq n^{-1/4}\sqrt{\beta_n}} \,  \big | ( \widetilde M_n(z) - M_n(z)  )_{i,j+r} \big|  \lesssim  n^{-\ell}  \,,  \label{410} \\ &
  \sup_{ |z - a_n| \leq n^{-1/4}\sqrt{\beta_n}} \,  \big | (  \widetilde M_n(z) - M_n(z) )_{i+r,j+r} \big|  \lesssim n^{-\ell} \sqrt{\beta_n} \,. \label{413}
\end{align}
\end{lemma}
\begin{proof} 
The argument is presented under assumptions (A1, A2, A3) for the diagonal of $\widetilde M_n(z) - M_n(z)$ only; the off-diagonal argument or argument assuming (B1, B2, A3) is similar and omitted. Consider $z$ satisfying $|z-a_n| \leq n^{-1/4} \sqrt{\beta_n}$ and $z' = a_n + i n^{-1/4} \sqrt{\beta_n}$.  We will decompose $\widetilde M_n(z) - M_n(z)$ as
\begin{align*}
\widetilde M_n(z) - M_n(z) &= \big( \widetilde M_n(z) - \widetilde M_n(z') \big) + \big( \widetilde M_n(z') - M_n(z') \big) +  \big(M_n(z') - M_n(z) \big) 
\end{align*}
and bound (the diagonal entries of) each term. 

\medskip
{\it Relation (\ref{411}): the upper left block}. For $1 \leq \alpha \leq n$, by the almost-sure eventual spacing bound (\ref{419}) and $|z-z'| \leq 2 n^{-1/4} \sqrt{\beta_n}$, 
\begin{align} \phantom{\,.}
   \bigg| \frac{1}{\lambda_\alpha - z} - \frac{1}{\lambda_\alpha - z'}  \bigg| =  \frac{|z-z'|}{| \lambda_\alpha - z||\lambda_\alpha - z'|} \leq \frac{8}{\eta^2 n^{1/4} \sqrt{\beta_n}} \, . \label{13579}
\end{align}
Hence, almost surely,
\begin{align}
    \Big|  \Big( \mfrac{1}{\sqrt{z}} \widetilde M_n(z) - \mfrac{1}{\sqrt{z'}} \widetilde M_n(z') \Big)_{i,i} \Big|& = \big| u_i^\top (S_n - z I_n)^{-1} u_i - u_i(S_n - z' I_n)^{-1} u_i\big|  \nonumber \\
    & \leq \|u_i\|_2^2 \, \big \| (S_n - z I_n)^{-1} - (S_n - z' I_n)^{-1} \big\|_2 \nonumber \\
     & = \|u_i\|_2^2 \bigg| \frac{1}{\lambda_1 - z} - \frac{1}{\lambda_1 - z'}  \bigg|   \lesssim \frac{1}{n^{1/4} \sqrt{\beta_n}} \, ,
\label{433}      \\
\Big|  \Big( \mfrac{1}{\sqrt{z}} M_n(z) - \mfrac{1}{\sqrt{z'}}  M_n(z') \Big)_{i,i} \Big| & = |s_n(z) - s_n(z')|  \leq \frac{1}{n} \sum_{\alpha=1}^n \bigg| \frac{1}{\lambda_\alpha - z} - \frac{1}{\lambda_\alpha - z'} \bigg| \lesssim \frac{1}{n^{1/4} \sqrt{\beta_n}} \label{002} \,.
\end{align}
Next, consider $\widetilde M_n(z') - M_n(z')$. Applying Lemma \ref{lem:xAx-tr} conditional on $X_n$, 
\begin{align} 
   \phantom{\,,} \E | u_i^\top (S_n - z' I_n)^{-1} u_i - s_n(z') |^4  \lesssim & \, \frac{1}{n^{4}} \E   | \tr (S_n - z'I_n)^{-1} (S_n - \bar z'I_n)^{-1} |^2 \nonumber\\ & + \frac{1}{n^{4}} \E \tr \big( (S_n - z'I_n)^{-1} (S_n - \bar z'I_n)^{-1} \big)^2   \nonumber \\
    \leq &\,  \frac{2}{n^{4}}\E \bigg(\sum_{\alpha=1}^n \frac{1}{|\lambda_\alpha - z'|^2} \bigg)^2 \, \nonumber  . 
\end{align}
By (\ref{111}), $|\lambda_\alpha - z'| \geq \eta \sqrt{\beta_n}/2$ with probability at least $1-o(n^{-1})$, while with probability one, $|\lambda_\alpha - z'| \geq \Im(z')$:
\begin{align*} 
   \phantom{\,,} \E | u_i^\top (S_n - z' I_n)^{-1} u_i - s_n(z') |^4  \lesssim & \, 
   \frac{1}{n^2} \bigg( \frac{1}{\beta_n^2} + \frac{1}{n\Im(z')^4} \bigg) \,.
\end{align*}
This yields the tail probability bound 
\begin{align} \text{Pr} \big( |u_i^\top (S_n - z' I_n)^{-1} u_i - s_n(z')| \geq n^{-\ell}\beta_n^{-1/2} \big) & \leq n^{4\ell } \beta_n^2 \E | u_i^\top (S_n - z' I_n)^{-1} u_i - s_n(z') |^{4}  \lesssim \frac{1}{n^{2-4\ell}}  \, , \nonumber
\end{align}
which is summable. Thus, by the Borel-Cantelli lemma, 
\begin{align} \phantom{\,,}
    \Big| \mfrac{1}{\sqrt{z'}}  \big( \widetilde M_n(z') - M_n(z') \big)_{i,i} \Big| & \lesssim \frac{1}{n^\ell \sqrt{\beta_n}} \, ,  \label{439}
\end{align}
almost surely. (\ref{411}) follows from (\ref{433}) - (\ref{439}), $|\sqrt{z'/z} | \leq 2$, and $|\sqrt{z'} - \sqrt{z}| \asymp n^{-1/4} \sqrt{\beta_n}$.

{\it Relation (\ref{413}): the lower right block}. Let $ \Breve{S}_n = \frac{1}{m} X_n^\top X_n$ denote the ``companion'' matrix to $S_n$, so-called because the eigenvalues of $\Breve{S}_n$ are the $n$ (possibly repeated) eigenvalues of $S_n$ joined by a zero eigenvalue with multiplicity $m-n$. Let $W$ be an orthogonal matrix such that $W^\top \Breve{S}_n W$ is diagonal. By Lemma \ref{lem:6.8}, we have an almost-sure bound on the magnitude of the component of $v_i$ orthogonal to the nullspace of $\Breve{S}_n$:
\begin{align} \phantom{\, .}   \|(W^\top v_i)_{1:n}\|_2^2 \lesssim \beta_n \log(n)  \,. \label{98765432}
\end{align}
Consider first $\widetilde M_n(z) - \widetilde M_n(z')$.
\begin{align}
  \Big|    \Big( \mfrac{1}{\sqrt{z}} \widetilde M_n(z) - \mfrac{1}{\sqrt{z'}} \widetilde M_n(z') \Big)_{i+r,i+r} \Big| & = \big |v_i^\top \big( (\Breve{S}_n - z I_m)^{-1} - (\Breve{S}_n- z' I_m)^{-1} \big) v_i \big|  \nonumber \\
 & \leq   \sum_{\alpha=1}^n (W^\top v_i)_\alpha^2 \bigg| \frac{1}{\lambda_\alpha - z} - \frac{1}{\lambda_\alpha - z'} \bigg|  + \sum_{\alpha=n+1}^m (W^\top v_i)^2_\alpha \bigg| \frac{1}{z} - \frac{1}{z'} \bigg| \nonumber 
    \\  & \leq \|(W^\top v_i)_{1:n}\|_2^2 \bigg| \frac{1}{\lambda_1 - z} - \frac{1}{\lambda_1 - z'} \bigg| + \|(W^\top v_i)_{n+1:m}\|_2^2\bigg| \frac{1}{z} - \frac{1}{z'}\bigg| \nonumber \, . 
\end{align}
The first term above is bounded using  (\ref{13579}) and (\ref{98765432}) and the second using $|1/z-1/z'| \leq 2 n^{-1/4} \sqrt{\beta_n}$: almost surely,
\begin{align}
    \big| \mfrac{1}{\sqrt{z}}  \big( \widetilde M_n(z) - \widetilde M_n(z') \big)_{i+r,i+r} \big| \lesssim \frac{\sqrt{\beta_n}}{n^{\ell}} \, .     \label{9599}
\end{align}
Next, consider $M_n(z) - M_n(z')$. By $(\ref{002})$,  
\begin{align}
  \Big|   \Big(  \mfrac{1}{\sqrt{z}} M_n(z) -   \mfrac{1}{\sqrt{z'}} M_n(z') \Big)_{i+r,i+r} \Big| & =\bigg| \beta_n s_n(z) -\frac{1-\beta_n}{z} - \beta_n s_n(z') +\frac{1-\beta_n}{z'} \bigg| \nonumber \\
   & \leq \beta_n | s_n(z) - s_n(z')| + (1-\beta_n) \bigg| \frac{1}{z} - \frac{1}{z'}\bigg|\lesssim \frac{ \sqrt{\beta_n}}{n^{1/4}} \,. \label{435} 
\end{align}
$\widetilde M_n(z') - M_n(z')$ is bounded similarly to (\ref{439}): applying Lemma \ref{lem:xAx-tr} conditional on $X_n$ and (\ref{111}), 
\begin{align*}
   \E \Big| v_i^\top (\Breve{S}_n - z' I_m)^{-1} v_i - \beta_n s_n(z') + \frac{1 - \beta_n}{z'} \Big|^{4} & \lesssim   \frac{1}{m^{4}} \E  \big| \tr (\Breve{S}_n - z' I_m)^{-1} (\Breve{S}_n - \bar z' I_m)^{-1} \big|^2\\
    &= \frac{1}{m^{4}} \E \bigg( \sum_{\alpha=1}^n \frac{1}{|\lambda_\alpha - z'|^2} + \frac{m-n}{|z'|^2}\bigg)^2  \\
   & \lesssim\frac{1}{m^{4}} \Big( \frac{n}{\beta_n} + m \Big)^2  \,, 
\end{align*}
 Thus,
\begin{align} \text{Pr} \Big(  \Big|  \mfrac{1}{\sqrt{z'}}  (\widetilde M_n(z') - M_n(z') )_{i+r,i+r} \Big| \geq n^{-\ell}\sqrt{\beta_n} \Big) \lesssim \frac{n^{4\ell}}{m^2\beta_n^2} \leq \frac{1}{n^{2-4\ell}}  \,.
\label{437}
\end{align}
Summability of the right-hand side, together with (\ref{9599}) and (\ref{435}), yields (\ref{413}).  The analysis of the off-diagonal entries of $\widetilde M_n(z)  - M_n(z)$ uses Lemma \ref{lem:xAy} rather than \ref{lem:xAx-tr}. Under assumptions B1, B2, the proof uses that the left and right singular vectors of $X_n$ are independent and each Haar-distributed.   
\end{proof}

\begin{lemma} \label{lem:M2} Let $a_n \geq 1+(2+\eta)\sqrt{\beta_n}$ be a bounded sequence. For any  $\ell < 1/4$ and $1 \leq i,j \leq r$, almost surely,  
\begin{align}  &  \sup_{ |z - a_n| \leq n^{-1/4}\sqrt{\beta_n} } \,  \Big | \frac{d}{dz} ( \widetilde M_n(z) - M_n(z)  )_{i,j} \Big|  \lesssim  n^{-\ell} \beta_n^{-1}   \,, \label{222} \\ &
  \sup_{ |z - a_n| \leq n^{-1/4}\sqrt{\beta_n}} \,  \Big | \frac{d}{dz}  (\widetilde M_n(z) - M_n(z)  )_{i,j+r} \Big|  \lesssim  n^{-\ell} \beta_n^{-1/2} \,, \label{223} \\ &
  \sup_{ |z - a_n| \leq n^{-1/4}\sqrt{\beta_n} } \,  \Big | \frac{d}{dz}  (  \widetilde M_n(z) - M_n(z) )_{i+r,j+r} \Big|  \lesssim n^{-\ell} \,.
\end{align}
Furthermore, for $i \neq j$, almost surely,
\begin{align}
       &
  \sup_{ |z - a_n| \leq n^{-1/4}\sqrt{\beta_n}} \,  \frac{1}{m}  \Big|  v_i^\top X_n^\top (S_n - z I_n)^{-2} X_n v_i -  \tr (S_n - z I_n)^{-2} S_n \Big| \lesssim n^{-\ell}  \,, \label{225} \\
   &
  \sup_{ |z - a_n| \leq n^{-1/4}\sqrt{\beta_n}} \, \frac{1}{m}  |  v_i^\top X_n^\top (S_n - z I_n)^{-2} X_n v_j | \lesssim n^{-\ell}   \,. \label{226}
\end{align}

\end{lemma}
\begin{proof}
The proof is similar to that of Lemma \ref{lem123} and is omitted. 
\end{proof}

\section{Proof of singular value results: (\ref{1.2}) and (\ref{1.5})}
The following lemmas, which we shall use to bound $|\det \widetilde M_n(z) - \det \overbar M_n(z)|$ and $\big| \frac{d}{dz} (\det \widetilde M_n(z) - \det \overbar M_n(z) ) \big|$ in the vicinities of $\bar \lambda_1, \ldots, \bar \lambda_{i_0}$, are critical to the proof of the theorem.  
\begin{lemma} \label{lem:detM-detM} Let $a_n \geq 1+(2+\eta)\sqrt{\beta_n}$ be a bounded sequence. For any $\ell < 1/4$, almost surely,
\begin{align}
\phantom{\,.}
   \sup_{|z-a_n| \leq n^{-1/4}\sqrt{\beta_n}} | \det \widetilde M_n(z) - \det \overbar M_n(z) | &\lesssim   n^{-\ell} \beta_n^{-r/2} \,. \label{333}
\end{align}
\end{lemma}
\begin{proof}
Partition $\widetilde M_n(z)$ into four blocks of size $r \times r$:
\begin{align*} \phantom{\,.}
   \widetilde M_n(z) =\begin{bmatrix}  \widetilde M_{11} & \widetilde M_{12} \\ \widetilde M_{21} & \widetilde M_{22}  \end{bmatrix} \, . 
\end{align*}
Let $\underline{\widetilde M}_n(z)$ denote the matrix containing the above four submatrices individually rescaled: 
\begin{align} \phantom{\,.}
   \underline{\widetilde M}_n(z) =\begin{bmatrix} \sqrt{\beta_n} \widetilde M_{11} & \beta_n^{1/4} \widetilde M_{12} \\ \beta_n^{1/4} \widetilde M_{21} & \widetilde M_{22}  \end{bmatrix} \, . \label{2834} 
\end{align}
Let $\underline M_n(z)$ and $\underline {\overbar M}_n(z)$ denote $M_n(z)$ and $\overbar M_n(z)$ similarly partitioned and rescaled. This rescaling simplifies comparison of determinants as all blocks are of similar size. Indeed, as consequences of Lemma \ref{lem:s_n bound}, Lemma \ref{lem123}, and Theorem \ref{thrm:s-s0} (stated in Section \ref{sec:b}), we have the following almost-sure results: 
\begin{align}
& \sup_{|z-a_n| \leq n^{-1/4}\sqrt{\beta_n}} \| \underline{\overbar M}_n(z) \|_\infty  \lesssim 1   \nonumber \, ,  &&  \sup_{|z-a_n| \leq n^{-1/4}\sqrt{\beta_n}} \| \underline{\widetilde M}_n(z) - \underline M_n(z) \|_\infty \lesssim n^{-\ell} \, , \nonumber \\ & \sup_{|z-a_n| \leq n^{-1/4}\sqrt{\beta_n}} \| \underline{M}_n(z) - \underline {\overbar M}_n(z) \|_\infty \lesssim n^{-\ell} \, , \label{13578}
\end{align}
where $\|\cdot\|_\infty$ is the maximum-magnitude entry of the input. 
Applying Lemma \ref{lem:detA-B}, we obtain
\[ \phantom{\,.} \sup_{|z-a_n| \leq n^{-1/4}\sqrt{\beta_n}} | \det \underline{\widetilde M}_n(z) - \det \underline{\overbar M}_n(z) | \lesssim n^{-\ell} \,,
\]
which is equivalent to (\ref{333}) as $\det \widetilde M_n(z) =  \beta_n^{-r/2} \det \underline{\widetilde M}_n(z)$. 
\end{proof}

\begin{lemma} \label{lem:d/dz detM-detM} Let $a_n \geq 1+(2+\eta)\sqrt{\beta_n}$ be a bounded sequence. For any $\ell < 1/8$, almost surely,
\begin{align*} \phantom{\,,} & \sup_{ |z - a_n| \leq n^{-1/4}\sqrt{\beta_n}} \Big| \frac{d}{dz} (\det \widetilde M_n(z) - \det \overbar M_n(z) ) \Big| \lesssim n^{-\ell}   \beta_n^{-(r+1)/2} \, , \qquad \\ &
\sup_{ |z - a_n| \leq n^{-1/4}\sqrt{\beta_n}} \Big| \frac{d}{dz} \det \widetilde M_n(z) \Big| \lesssim \beta_n^{-(r+1)/2} \, . 
\end{align*}
\end{lemma}
\begin{proof} The proof is similar to that of Lemma \ref{lem:detM-detM}. By Lemma \ref{lem:s_n bound}, Lemma \ref{lem:M2}, and Corollary \ref{cor:s'-s0'},
\begin{align}
    &  \sup_{ |z - a_n| \leq n^{-1/4}\sqrt{\beta_n}}   \Big \| \frac{d}{dz} \underline{\overbar M}_n(z) \Big\|_\infty \lesssim \beta_n^{-1/2}  \, , \\ &  \sup_{ |z - a_n| \leq n^{-1/4}\sqrt{\beta_n}}  \Big\| \frac{d}{dz} ( \underline{\widetilde M}_n(z) -  \underline{M}_n(z) ) \Big\|_\infty \lesssim n^{-\ell} \beta_n^{-1/2} \, , \nonumber  \\
    &  \sup_{ |z - a_n| \leq n^{-1/4}\sqrt{\beta_n}}  \Big\| \frac{d}{dz} ( \underline{ M}_n(z) - \underline{\overbar M}_n(z) ) \Big\|_\infty \lesssim n^{-\ell}  \beta_n^{-1/2} \, . \label{13577}
\end{align}
The claim follows from Lemma \ref{lem:d/dz detA-detB}, (\ref{13578}), and (\ref{13577}):
\begin{align*}
  \Big| \frac{d}{dz} (\det \underline{\widetilde M}_n(z) - \det \underline{\overbar M}_n(z) ) \Big|&  \lesssim   \Big\| \frac{d}{dz} ( \underline{\widetilde M}_n(z) -  \underline{\overbar M}_n(z) ) \Big\|_\infty \\ &+  \Big\| \frac{d}{dz} \underline{\overbar M}_n(z) \Big\|_\infty  \| \underline{\widetilde M}_n(z) - \underline M_n(z) \|_\infty  
 \lesssim n^{-\ell} \beta_n^{-1/2} \, . \end{align*}
\end{proof}

\begin{proof}[Proof of (\ref{1.2}) and (\ref{1.5})] We initially assume $\varepsilon_{i,n} = 0$, $1 \leq i \leq r$.  Fix $\ell < 1/4$ and let $\gamma_{i,n}$ denote a circular contour with center $\bar \lambda_i$ and radius $n^{-\ell} \sqrt{\beta_n}$. We divide the proof into four claims:
\begingroup
\begin{enumerate}[label=(\roman*)]
\item $\gamma_{i,n}$ eventually encircles a single root of $\det \overbar M_n(z)$.
\item Almost surely eventually, there exists a unique root $\tilde \lambda_{(i)}$ of $\det \widetilde M_n(z)$ encircled by $\gamma_{i,n}$. 
\item  $\tilde \lambda_{(1)}, \ldots, \tilde \lambda_{(i_0)}$ are real and therefore eigenvalues of $\widetilde S_n$ by Lemma \ref{lem:4.1BGN}.
\item Almost surely eventually, $\widetilde S_n$ has no eigenvalues other than  $\tilde \lambda_{(1)}, \ldots, \tilde \lambda_{(i_0)}$ larger than $1+(2+\eta)\sqrt{\beta_n}$, where $\eta > 0$ is arbitrary. 
\end{enumerate}
\endgroup

Notice that $\gamma_{i,n}$ eventually encircles a single real root of $\det \overbar M_n(z)$: as $\tau_1, \ldots, \tau_{i_0}$ are distinct,  $|\bar \lambda_i - \bar \lambda_j| \gtrsim \sqrt{\beta_n}$, $i \neq j$. We now argue $\gamma_{i,n}$ eventually encircles no complex roots of $\det \overbar M_n(z)$. For  $\xi \in \mathbb{C}$ such that $|\xi| \leq n^{-\ell} \sqrt{\beta_n}$,
\begin{align*}  \overbar D_n(\bar \lambda_i + \xi) - \theta_i^{-2} &= \frac{\xi + \sqrt{(\overbar \lambda_i - \beta_n-1)^2 - 4 \beta_n} (1 - \sqrt{1+y})}{2 \beta_n} \, , \\ 
y & \coloneqq \frac{2(\bar \lambda_i - \beta_n-1) \xi + \xi^2}{(\bar \lambda_i - \beta_n-1)^2-4\beta_n} = \frac{2 (\tau_i^2 + \tau_i^{-2})  \beta_n^{-1/2} \xi + \beta_n^{-1}\xi^2 }{(\tau_i^2 + \tau_i^{-2})^2-4} \,.
\end{align*}
Using $|1-\sqrt{1+z}+z/2| \leq |z|^2/8$ for $z \in \{\Re(z) \geq 0, |z| \leq 1\}$, $|\xi| \leq n^{-\ell} \sqrt{\beta_n}$, and  $|y| \lesssim n^{-\ell}$, 
\begin{align} \phantom{\,,} | \overbar D_n(\bar \lambda_i + \xi) - \theta_i^{-2}| &  \asymp |\xi|\beta_n^{-1} \,. \label{77} \end{align}
For $j \neq i$, almost surely,
\begin{align} | \overbar D_n(\bar \lambda_i + \xi) - \theta_j^{-2} | &= | \overbar D_n(\bar \lambda_i + \xi) - \theta_i^{-2} + (\theta_i^{-2} - \theta_j^{-2})|\asymp \beta_n^{-1/2}  \,. \label{777} \end{align} 
Recalling the formula $\det \overbar M_n(z) = \prod_{j=1}^r (\overbar D_n(z) - \theta_j^{-2} )$, (\ref{77}) and (\ref{777}) imply
\begin{align}
 \phantom{\,.} | \det \overbar M_n(\bar \lambda_i + \xi)| \asymp  |\xi|\beta_n^{-(r+1)/2}  \,.
\label{77777}\end{align}
Consequently, $\gamma_{i,n}$ eventually encircles no root of $\det \overbar M_n(z)$ other than $\bar \lambda_i$ (and no roots occur on $\gamma_{i,n}$). Otherwise, there exists a sequence of perturbations $\xi_n \neq 0$ such that $\det \overbar M_n(\bar \lambda_i + \xi_n) = 0$, contradicting (\ref{77777}). This proves Claim (i). As a consequence, by the winding number theorem,
\begin{align}
    \phantom{\,.} \frac{1}{2 \pi i} \int_{\gamma_{i,n}} \frac{\frac{d}{dz} \det \overbar M_n(z)}{\det \overbar M_n(z)} dz = 1 \label{2374} \,.
\end{align}

To prove Claim (ii), it suffices to show that almost surely eventually, 
\begin{align*}
    \phantom{\,.} \bigg| \frac{1}{2 \pi i} \int_{\gamma_{i,n}} \frac{\frac{d}{dz} \det \widetilde M_n(z)}{\det \widetilde M_n(z)} dz  - \frac{1}{2 \pi i} \int_{\gamma_{i,n}} \frac{\frac{d}{dz} \det \overbar M_n(z)}{\det \overbar M_n(z)} dz \bigg| < 1 \,,
\end{align*}
as the above integrals are integer-valued. Writing $\frac{a}{c} - \frac{b}{d} = \frac{1}{d}\big( \frac{d-c}{c}a + a-b \big)$, an upper bound on the left-hand side is
\begin{align}
& \, n^{-\ell} \sqrt{\beta_n} \sup_{z \in \gamma_{i,n}}  \bigg| \frac{\frac{d}{dz}\det  \widetilde M_n(z)}{\det \widetilde M_n(z)}  - \frac{\frac{d}{dz} \det \overbar M_n(z)}{\det \overbar M_n(z)} \bigg| \\  = &  n^{-\ell} \sqrt{\beta_n}    \sup_{z \in \gamma_{i,n}} \bigg| \frac{1}{\det \overbar M_n} \bigg| \bigg| \frac{\det \overbar M_n - \det \widetilde M_n}{\det \widetilde M_n} \cdot \frac{d}{dz} \det \widetilde M_n  + \frac{d}{dz} \big(  \det \widetilde M_n -  \det \overbar M_n \big) \bigg|  \,. \label{7771} 
\end{align}
Below, the almost-sure bounds we have developed on the terms of (\ref{7771}) are summarized. Let $\ell' \in (\ell, 1/4)$ and $\ell'' \in (0,1/8)$. By Lemmas \ref{lem:detM-detM} and \ref{lem:d/dz detM-detM} (applied to $a_n = \bar \lambda_i$ and $\eta \in (0, \tau_{i_0}^2 + \tau_{i_0}^{-2} - 2)$), 
\begin{align}
 \phantom{\,.}   &  \sup_{z \in \gamma_{i,n}} | \det \widetilde M_n(z) - \det \overbar M_n(z) | \lesssim   n^{-\ell'} \beta_n^{-r/2} \,, \label{7770} \\ 
 \phantom{\, . } & \sup_{z \in \gamma_{i,n}} \Big| \frac{d}{dz} (\det \widetilde M_n(z) - \det \overbar M_n(z) ) \Big| \lesssim n^{-\ell''} \beta_n^{-(r+1)/2} \, , \\
 \phantom{\,.} & \sup_{z \in \gamma_{i,n}} \Big| \frac{d}{dz} \det \widetilde M_n(z) \Big| \lesssim \beta_n^{-(r+1)/2} \, . 
\end{align}
As (\ref{77777}) depends on $\xi$ only through $|\xi|$,
\begin{align}
    &  |\det \overbar M_n(z) | \asymp n^{-\ell} \beta_n^{-r/2} \, .
\label{7772}\end{align} 
Finally, (\ref{7770}), (\ref{7772}), and $\ell < \ell'$ imply,
\begin{align}
&   |\det \widetilde M_n(z) | \asymp n^{-\ell} \beta_n^{-r/2} \, . \label{7773}
\end{align}
Thus, from (\ref{7771}) - (\ref{7773}), almost surely,
\begin{align}
 n^{-\ell} \sqrt{\beta_n} \sup_{z \in \gamma_{i,n}}  \bigg| \frac{\frac{d}{dz}\det  \widetilde M_n(z)}{\det \widetilde M_n(z)}  - \frac{\frac{d}{dz} \det \overbar M_n(z)}{\det \overbar M_n(z)} \bigg| 
     \lesssim & \,  \frac{n^{-\ell} \sqrt{\beta_n}}{n^{-\ell}\beta_n^{-r/2} } \bigg( \frac{n^{-\ell'} \beta_n^{-r/2} }{n^{-\ell} \beta_n^{-r/2}} \cdot \beta_n^{-(r+1)/2} +  n^{-\ell''}\beta_n^{-(r+1)/2}  \bigg) \nonumber \\
     \lesssim & \, n^{\ell-\ell'} + n^{-\ell''} \longrightarrow 0         \,. \label{006}
\end{align}
\noindent This completes the proof of Claim (ii): almost surely eventually, there exist roots $\tilde \lambda_{(1)}, \ldots, \tilde \lambda_{(i_0)}$ respectively encircled by $\gamma_{1,n}, \ldots, \gamma_{i_0, n}$. Claim (iii)---that $\tilde \lambda_{(1)}, \ldots, \tilde \lambda_{(i_0)}$ are eigenvalues of $\widetilde S_n$---follows from  Lemma \ref{lem:4.1BGN} and the fact that $\widetilde M_n(z)$ is invertible for $z \not \in \mathbb{R}$ (Lemma A.1 of \cite{BGN12}).

It remains to prove Claim (iv): $\tilde \lambda_1 = \tilde \lambda_{(1)}, \ldots, \tilde \lambda_{i_0} = \tilde \lambda_{(i_0)}$.  
 Suppose $\det \widetilde M_n(z)$ has a sequence of roots $\tilde \lambda \in [1+(2+\eta)\sqrt{\beta_n}, C]$, for some $\eta > 0$, $C > 1$, such that infinitely often, $|\tilde \lambda - \bar \lambda_i| > n^{-\ell} \sqrt{\beta_n}$ for all $i \leq i_0$. Let $\gamma_n$ denote a sequence of circular contours of radius $n^{-\ell} \sqrt{\beta_n}$ centered at $\tilde \lambda$. Then, eventually,
\begin{align}
   & \frac{1}{2 \pi i} \int_{\gamma_n} \frac{\frac{d}{dz} \det \widetilde M_n(z)}{\det \widetilde M_n(z)} dz = 1  \,, 
   &   \frac{1}{2 \pi i} \int_{\gamma_n} \frac{\frac{d}{dz} \det \overbar M_n(z)}{\det \overbar M_n(z)} dz = 0 \,. \label{51}
\end{align}
Note that on $\gamma_n$, (\ref{7770}) - (\ref{7772}) hold while (\ref{7773}) becomes $\beta_n^{-r/2}$.  Thus, similarly to (\ref{006}), the left integral in (\ref{51}) converges to the right integral, a contradiction. Furthermore, eventually $\widetilde S_n$ has no eigenvalues greater than $C$:
\[ \phantom{\,.} \frac{1}{\sqrt{m}} \|\widetilde X_n\|_2 \leq \sum_{i=1}^r \theta_i \|u_i\|_2 \|v_i\|_2 + \frac{ 1}{\sqrt{m}} \|X_n\|_2 \leq C \,,
\]
the last inequality holding almost surely eventually. This establishes Claim (iv). We conclude that almost surely eventually, the leading eigenvalues of $\widetilde S_n$ satisfy
\begin{align*} &  |\tilde  \lambda_i - \bar \lambda_i| \leq  n^{-\ell} \sqrt{\beta_n}   \,, \hspace{2cm}  i \leq i_0 \, , \end{align*}
which is (\ref{1.5}) in the case $\varepsilon_{i,n} = 0$. 

For general signal strengths $\theta_i = \tau \beta_n^{1/4} (1+\varepsilon_{i,n})$, the above argument may be used to show that $\tilde \lambda_i$ localizes about  $\overbar D_n^{(-1)} ( \theta_i^{-2} )$ and that Claim (iv) holds. Recalling that $\overbar D_n^{(-1)}(t)   = (t+1)(\beta_n t+1)/t$ (Lemma \ref{lem:Dbar}), for $i \leq i_0$,
\begin{align} \phantom{\, .} \big|  \overbar D_n^{(-1)} ( \theta_i^{-2} ) - \bar \lambda_i \big| &=  \big|  \overbar D_n^{(-1)} ( \theta_i^{-2} ) - \overbar D_n^{(-1)}\big( (\tau_i \beta_n^{1/4})^{-2}\big)  \big| = (\tau_i^2 + \tau_i^{-2} (1+\varepsilon_{i,n})^{-2} ) | \varepsilon_{i,n} | (2+\varepsilon_{i,n}) \sqrt{\beta_n}  \nonumber
\\ &\lesssim |\varepsilon_{i,n}| \sqrt{\beta_n}        \, .
\end{align}
Thus, almost surely eventually,
\begin{align*} \phantom{\,.} |\tilde \lambda_i - \bar \lambda_i| &\leq   \big| \tilde \lambda_i - \overbar D_n^{(-1)} ( \theta_i^{-2} )  \big| +  \big|\overbar D_n^{(-1)} ( \theta_i^{-2} ) - \bar \lambda_i \big|     \lesssim (n^{-\ell} + |\varepsilon_{i,n}| ) \sqrt{\beta_n} \, ,
\end{align*} 
establishing (\ref{1.5}). (\ref{1.2}) is a consequence of (\ref{1.5}) and Claim (iv).

\end{proof}

\section{Proof of singular vector results: (\ref{qwerty}), (\ref{1.8}), and (\ref{1.9})}
We prove (\ref{1.8}) and (\ref{1.9}), from which (\ref{qwerty}) immediately follows.   For notational lightness, this section assumes $\varepsilon_{i,n} = 0$. All bounds are understood to hold only almost surely.
\begin{lemma} Let $i \leq i_0$, $j \leq r$, and $\ell < 1/4$. Let $\tilde u_i, \tilde v_i$ be the singular vectors corresponding to $\tilde \sigma_i = \sqrt{\tilde \lambda_i}$. Almost surely,
\begin{align}
  &  \phantom{\,.}  | \bar \sigma_i \bar s_n(\bar \lambda_i) \theta_j \langle \tilde v_i, v_j \rangle  +  \langle \tilde u_i, u_j \rangle | \lesssim  n^{-\ell} ( \beta_n^{-1/4}  \| V_n^\top \tilde v_i \|_1 + \beta_n^{1/4} )  \label{912} \, , \end{align}
and for $i \neq j$,
\begin{align}
    \phantom{\,.} | \langle \tilde v_i, v_j \rangle |  \lesssim n^{-\ell}  ( | \langle \tilde v_i, v_i \rangle | +  \sqrt{\beta_n} ) \, . \label{916}
\end{align}
\end{lemma}

\begin{proof}
 Recall Lemma \ref{lem: BG5.1}: the column vector \begin{align} \begin{bmatrix} \Theta_n V_n^\top \tilde v_i \\ \Theta_n U_n^\top \tilde u_i \end{bmatrix}    \label{692} \end{align}
belongs to the kernel of $\widetilde M_n(\tilde \lambda_i)$. As (\ref{692}) is orthogonal to the $j$-th row of $\widetilde M_n(\tilde \lambda_i)$,
\begin{align*}
 \phantom{\,.}    \sum_{k=1}^r \tilde \sigma_i u_j^\top (S_n - \tilde \lambda_i I_n)^{-1} u_k \theta_k \langle \tilde v_i, v_k \rangle + \sum_{k=1}^r \Big( \theta_k^{-1} {\bf 1}_{\{j = k\}} + \frac{1}{\sqrt{m}} u_j^\top (S_n - \tilde \lambda_i I_n)^{-1} X_n v_k \Big) \theta_k \langle \tilde u_i, u_k \rangle   = 0 \,.
\end{align*}
Isolating the $j$-th term, by Lemma \ref{lem123}, 
\begin{align}  
| \tilde \sigma_i u_j^\top (S_n - \tilde \lambda_i I_n)^{-1} u_j \theta_j \langle \tilde v_i, v_j \rangle  +  \langle \tilde u_i, u_j \rangle | & \leq \tilde \sigma_i \sum_{k \neq j} \theta_k  | u_j^\top (S_n - \tilde \lambda_i I_n)^{-1} u_k | | \langle \tilde v_i, v_k \rangle |\nonumber \\ & + \frac{1}{\sqrt{m}} \sum_{k=1}^r \theta_k | u_j^\top (S_n - \tilde \lambda_i I_n)^{-1} X_n v_k | 
\nonumber   \\
&\lesssim n^{-\ell} \beta_n^{-1/4}  \sum_{k \neq j} | \langle \tilde v_i, v_k \rangle | + n^{-\ell} \beta_n^{1/4} \, , \end{align} 
\begin{align}
  | \tilde \sigma_i s_n(\tilde \lambda_i) \theta_j \langle \tilde v_i, v_j \rangle  +  \langle \tilde u_i, u_j \rangle | &  \leq | \tilde \sigma_i u_j^\top (S_n - \tilde \lambda_i I_n)^{-1} u_j \theta_j \langle \tilde v_i, v_j \rangle \nonumber\\ & +  \langle \tilde u_i, u_j \rangle |  + \tilde \sigma_i \theta_j |u_j^\top (S_n - \tilde \lambda_i I_n)^{-1} u_j  - s_n(\tilde \lambda_i) | | \langle \tilde v_i, v_j \rangle | \nonumber\\
  & \lesssim  n^{-\ell} \beta_n^{-1/4}  \sum_{k =1}^r | \langle \tilde v_i, v_k \rangle | + n^{-\ell} \beta_n^{1/4}  \, . \label{912.0}
\end{align}
Note that eventually, $|s_n(\tilde \lambda_i) - s_n(\bar \lambda_i)| \leq \frac{|\tilde \lambda_i - \bar \lambda_i|}{| \tilde \lambda_i - \lambda_1||\bar \lambda_i - \lambda_1|}$. By an application of (\ref{1.5}) and  Theorem \ref{thrm:s-s0}, 
\begin{align}  | \tilde \sigma_i s_n(\tilde \lambda_i) - \bar \sigma_i \bar s_n(\bar \lambda_i)| & \leq |\tilde \sigma_i - \bar  \sigma_i|   |s_n(\tilde \lambda_i)| + | \bar  \sigma_i | | s_n(\tilde \lambda_i) - s_n(\bar \lambda_i)| + | \bar  \sigma_i| | s_n(\bar \lambda_i) - \bar s_n(\bar \lambda_i)| \nonumber \\
\phantom{\,.} & \lesssim  n^{-\ell} + n^{-\ell}  \beta_n^{-1/2} + n^{-\ell}  \beta_n^{-1/2}  \,. \label{912.1}
\end{align}
(\ref{912}) follows from (\ref{912.0}) and (\ref{912.1}).
By a similar argument for row $j+r$ of $\widetilde M_n(\tilde \lambda_i)$, 
\begin{align} \phantom{\,.} \Big| \langle \tilde v_i, v_j \rangle  + \bar \sigma_i \Big( \beta_n \bar s_n(\bar \lambda_i) - \frac{1-\beta_n}{\bar \lambda_i}\Big) \theta_j \langle \tilde u_i, u_j \rangle \Big| \lesssim  n^{-\ell} (  \beta_n^{1/4}  \| V_n^\top \tilde v_i \|_1  +  \beta_n^{3/4} )   \label{913} \,.
\end{align}
Substitution of $(\ref{912})$ into $(\ref{913})$ yields
\[ \phantom{\,.} \Big|  \langle \tilde v_i, v_j \rangle - \bar \lambda_i \bar s_n(\bar \lambda_i)   \Big(\beta_n \bar s_n(\bar \lambda_i) - \frac{1-\beta_n}{\bar \lambda_i}\Big) \theta_j^2 \langle \tilde v_i, v_j \rangle  \Big| \lesssim 
n^{-\ell} ( \| V_n^\top \tilde v_i \|_1  + \sqrt{\beta_n} )  \,.\]
Recall that $ \overbar D_n(z) = \beta_n z \bar s_n(z)^2 - (1-\beta_n) \bar s_n(z)$ and $\bar \lambda_i$ is defined by  $\overbar D_n(\bar \lambda_i) = \theta_i^{-2}$. Thus, for $j \neq i$,
\begin{align}
   \phantom{\,.}  | \langle \tilde v_i, v_j \rangle | \lesssim  n^{-\ell} ( \| V_n^\top \tilde v_i \|_1  + \sqrt{\beta_n} )  \ \,. \nonumber
\end{align}
Rewriting this as $ (1-Cn^{-\ell}) (\|V_n^\top \tilde v_i\|_1 - | \langle \tilde v_i, v_i \rangle |) \lesssim n^{-\ell} ( | \langle \tilde v_i, v_i \rangle | + \sqrt{\beta_n})$, we  recover (\ref{916}):
\begin{gather*} \phantom{\,.} | \langle \tilde v_i, v_j \rangle | \lesssim n^{-\ell} ( | \langle \tilde v_i, v_i \rangle | + \sqrt{\beta_n}) \, . 
\end{gather*}
\end{proof}

\begin{proof} [Proof of  (\ref{1.8}) and (\ref{1.9})]
\noindent Let $i \leq i_0$ and $\ell < 1/8$. By (\ref{765}), $a_n + b_n + 2 c_n = 1$ with 
\begin{align}
 \phantom{\,.} a_n &= \tilde \lambda_i \tilde v_i^\top P_n^\top (S_n - \tilde \lambda_i)^{-2} P_n \tilde v_i = \tilde \lambda_i \sum_{j,k=1}^r \theta_j \theta_k \langle \tilde v_i, v_j \rangle \langle \tilde v_i, v_k \rangle u_j^\top (S_n - \tilde \lambda_i I_n)^{-2} u_k \nonumber \,, \\
\phantom{\,.} b_n &= \frac{1}{m}\tilde u_i^\top P_n X_n^\top (S_n - \tilde \lambda I_n)^{-2} X_n P_n^\top \tilde u_i
= \frac{1}{m} \sum_{j,k=1}^r  \theta_j \theta_k \langle \tilde u_i, u_j \rangle \langle \tilde u_i, u_k \rangle v_j^\top X_n^\top (S_n - \tilde \lambda_i I_n)^{-2} X_n v_k  \nonumber \,, \\
\phantom{\,.} c_n &=  \frac{\tilde \sigma_i}{\sqrt{m}} \tilde u_i^\top P_n X_n^\top ( S_n - \tilde \lambda I_n)^{-2} P_n \tilde v_i = \frac{\tilde \sigma_i}{\sqrt{m}} \sum_{j,k=1}^r  \theta_j \theta_k \langle \tilde u_i, u_j \rangle \langle \tilde v_i, v_k \rangle v_j^\top X_n^\top (S_n - \tilde \lambda_i I_n)^{-2}  u_k 
\,. 
\end{align}
By (\ref{1.5}) and Corollary \ref{cor:s'-s0'},
\begin{align*}
    \Big| \tilde \lambda_i \frac{d}{dz} s_n(\tilde \lambda_i) - \bar \lambda_i \frac{d}{dz} \bar s_n(\bar \lambda_i) \Big| & \leq  |\tilde \lambda_i - \bar \lambda_i | \Big| \frac{d}{dz} s_n(\tilde \lambda_i) \Big| + | \bar \lambda_i | \Big| \frac{d}{dz}(s_n(\tilde \lambda_i) - s_n(\bar \lambda_i)) \Big| + | \bar \lambda_i | \Big|  \frac{d}{dz}(s_n(\bar \lambda_i) - \bar s_n(\bar \lambda_i)) \Big| \\
  & \lesssim   n^{-2\ell} \beta_n^{-1/2}  + n^{-2\ell} \beta_n^{-1}  +  n^{-\ell} \beta_n^{-1}  \,.
\end{align*}
Together with (\ref{222}) of Lemma \ref{lem:M2},
\begin{align}
  a_n &= \sum_{j,k=1}^r\theta_j \theta_k \langle \tilde v_i, v_j \rangle \langle \tilde v_i, v_k \rangle \Big( \delta_{jk} \bar \lambda_i \frac{d}{dz} \bar s_n(\bar \lambda_i) + O( n^{-\ell} \beta_n^{-1}) \Big)  \nonumber\\
  & = \sum_{j=1}^r \theta_j^2 | \langle \tilde v_i, v_j \rangle |^2  \bar \lambda_i \frac{d}{dz} \bar s_n(\bar \lambda_i)  + O(n^{-\ell} \beta_n^{-1/2}  \|V_n^\top \tilde v_i \|_1^2 )\, .
    \end{align}
Similarly, using  
\begin{align*}
      \phantom{\, ,}
      \frac{1}{m} \tr (S_n - \tilde \lambda_i)^{-2} S_n & = \frac{1}{m} \tr \big( (S_n - \tilde \lambda_i I_n)^{-1} + \tilde \lambda_i (S_n - \tilde \lambda_i I_n)^{-2} \big)   = \beta_n \Big( s_n(\tilde \lambda_i) + \tilde \lambda_i \frac{d}{dz} s_n(\tilde \lambda_i) \Big)   \,,
\end{align*}
(\ref{223}), (\ref{225}), and (\ref{226}),  it may be shown that
\begin{align}
   b_n &= 
\sum_{j,k=1}^r  \theta_j \theta_k \langle \tilde u_i, u_j \rangle \langle \tilde u_i, u_k \rangle  \Big( \delta_{jk} \beta_n  \Big( \bar s_n(\bar \lambda_i) + \bar \lambda_i \frac{d}{dz} \bar s_n(\bar \lambda_i) \Big) + O(n^{-\ell}) \Big)  \nonumber \\
&=  \sum_{j=1}^r  \theta_j^2 | \langle \tilde u_i, u_j \rangle |^2 \beta_n  \Big( \bar s_n(\bar \lambda_i) + \bar \lambda_i \frac{d}{dz} \bar s_n(\bar \lambda_i) \Big)  + O(n^{-\ell} \sqrt{\beta_n}) \nonumber \, , \\
\phantom{\,.} |c_n| & \lesssim n^{-\ell} \beta_n^{-1/2}  \sum_{j,k=1}^r  \theta_j \theta_k | \langle \tilde u_i, u_j \rangle \langle \tilde v_i, v_k \rangle |  \lesssim n^{-\ell} \, .
\end{align}
Thus, 
\begin{align}
   &  \sum_{j=1}^r \theta_j^2 | \langle \tilde v_i, v_j \rangle |^2  \bar \lambda_i \frac{d}{dz} \bar s_n(\bar \lambda_i) +  \sum_{j=1}^r  \theta_j^2 | \langle \tilde u_i, u_j \rangle |^2 \beta_n  \Big( \bar s_n(\bar \lambda_i) + \bar \lambda_i \frac{d}{dz} \bar s_n(\bar \lambda_i) \Big)  =  1 + O\big(n^{-\ell}(1 +  \beta_n^{-1/2}   \| V_n^\top \tilde v_i \|_1^2) \big) \,. \label{555}
\end{align}
 Terms on the left-hand side with $j \neq i$ are negligible: by (\ref{916}), 
\begin{align*}
   & \bigg| \sum_{j \neq i} \theta_j^2 | \langle \tilde v_i, v_j \rangle |^2  \bar \lambda_i \frac{d}{dz} \bar s_n(\bar \lambda_i) \bigg| \lesssim  n^{-4 \ell}  ( \beta_n^{-1/2} | \langle \tilde v_i, v_i \rangle |^2 + \sqrt{\beta_n}) \, , \nonumber \\
   &     \bigg|  \sum_{j \neq i}  \theta_j^2 | \langle \tilde u_i, u_j \rangle |^2 \beta_n  \Big( \bar s_n(\bar \lambda_i) + \bar \lambda_i \frac{d}{dz} \bar s_n(\bar \lambda_i) \Big)   \bigg|  \lesssim \sqrt{\beta_n} \, ,
\end{align*}
implying
\begin{align}
     & \theta_i^2 | \langle \tilde v_i, v_i \rangle |^2  \bar \lambda_i \frac{d}{dz} \bar s_n(\bar \lambda_i) +  \theta_i^2 | \langle \tilde u_i, u_i \rangle |^2 \beta_n \Big( \bar s_n(\bar \lambda_i) + \bar \lambda_i \frac{d}{dz} \bar s_n(\bar \lambda_i) \Big) =  1 + O\big(n^{-\ell}(1 +  \beta_n^{-1/2}   \| V_n^\top \tilde v_i \|_1^2) + \sqrt{\beta_n} \big) \, . \label{557} 
\end{align}

    Substitution of (\ref{912}) into (\ref{557}) yields 
\begin{align}
   &  | \langle \tilde v_i, v_i \rangle |^2 \Big( \theta_i^2  \bar \lambda_i \frac{d}{dz} \bar s_n(\bar \lambda_i) +  \theta_i^4 \bar \lambda_i \bar s_n(\bar \lambda_i)^2    \beta_n  \Big( \bar s_n(\bar \lambda_i) + \bar \lambda_i \frac{d}{dz} \bar s_n(\bar \lambda_i) \Big) \Big) \nonumber \\
    = \, &  | \langle \tilde v_i, v_i \rangle |^2 \bigg( \theta_i^2 \int \frac{\bar \lambda_i}{(t - \bar \lambda_i)^2} d \overbar F_n(t) + \theta_i^4 \bar \lambda_i \bar s_n(\bar \lambda_i)^2 \beta_n \int \frac{t}{(t-\bar \lambda_i)^2} d \overbar F_n(t) \bigg) \nonumber \\ 
     = \, &    1 + O\big(n^{-\ell}(1 +  \beta_n^{-1/2}   \| V_n^\top \tilde v_i \|_1^2) + \sqrt{\beta_n} \big)  \, . \label{851}
\end{align}
 (\ref{851}) is included to facilitate comparison with Theorem 2.10 of \cite{BGN12}. In the regime where $\beta_n \rightarrow \beta > 0$ and $\theta_i > 1$ is constant, the term multiplying $|\langle \tilde v_i, v_i \rangle|^2$ converges to a positive limit, and $|\langle \tilde v_i, v_i \rangle|^2$ converges to the limit's inverse. Here, since 
 \[  \phantom{}\Big|\theta_i^2 | \langle \tilde u_i, u_i \rangle |^2 \beta_n \Big( \bar s_n(\bar \lambda_i) + \bar \lambda_i \frac{d}{dz} \bar s_n(\bar \lambda_i) \Big) \Big| \lesssim \sqrt{\beta_n} \,, \]
(\ref{916}) and (\ref{557}) imply
\begin{align}  \phantom{\,.} \theta_i^2 | \langle \tilde v_i, v_i \rangle |^2   \bar \lambda_i \frac{d}{dz} \bar s_n(\bar \lambda_i) = 1 + O\big(n^{-\ell}(1 +  \beta_n^{-1/2}   \| V_n^\top \tilde v_i \|_1^2) + \sqrt{\beta_n} \big) \label{849} 
\, . \end{align}
Evaluated at $z = \bar \lambda_i$, Lemma \ref{lem:s_n bound} is tight up to constants: $\bar \lambda_i \frac{d}{dz} \bar s_n(\bar \lambda_i) \asymp \beta_n^{-1}$. Thus, for $j \neq i$,
\begin{align}
  &  (1 - C n^{-\ell} )  | \langle \tilde v_i, v_i \rangle |^2   \lesssim \sqrt{\beta_n} \,, & | \langle \tilde v_i, v_i \rangle | \lesssim \beta_n^{1/4} \,, && |\langle \tilde v_i, v_j \rangle | \lesssim n^{-2 \ell} \beta_n^{1/4} \,. \label{848}
\end{align}
The third inequality follows from (\ref{916}). We have established (\ref{1.9}): right singular vectors asymptotically do not correlate with right signal vectors. 

Now, by (\ref{912}), (\ref{849}), and (\ref{848}), we obtain
\begin{align}
     |\langle \tilde u_i, u_i \rangle|^2 &= ( \bar \sigma_i \bar s_n(\bar \lambda_i) \theta_i \langle \tilde v_i, v_i \rangle )^2 + O(n^{-4\ell}) 
     = \frac{\bar s_n(\bar \lambda_i)^2 }{\frac{d}{dz} \bar s_n(\bar \lambda_i)}   + O(n^{-\ell} + \sqrt{\beta_n})  \, . \label{847}
\end{align}
Direct calculation yields 
\begin{align*}
    & \bar s_n(\bar \lambda_i) = \frac{-(\tau_i^2 + \tau_i^{-2}) + \sqrt{(\tau_i^2 + \tau_i^{-2})^2 - 4}}{2 \sqrt{\beta_n}} + O(1) \, , \\
    & \frac{d}{dz} \bar s_n(\bar \lambda_i) = -\frac{1}{2 \beta_n \bar \lambda_i} \bigg( 1 - \frac{\tau_i^2 + \tau_i^{-2} }{\sqrt{(\tau_i^2 + \tau_i^{-2})^2-4}} \bigg) + O(\beta_n^{-1/2})  \, , \end{align*}
from which we conclude (\ref{1.8}):
\begin{align}
    \phantom{\,.} |\langle \tilde u_i, u_i \rangle |^2 = 1 - \tau_i^4 + O(n^{-\ell} + \sqrt{\beta_n}) \,.
\end{align}
Therefore, $|\langle \tilde u_i, u_i \rangle |^2  \xrightarrow{a.s.}  1 - \tau_i^{-4}$. For $j \neq i$, $| \langle \tilde u_i, u_j \rangle | \lesssim n^{-2\ell}$ follows from (\ref{912}).  

\end{proof} 


\section{Convergence rate of the Stieltjes transform} \label{sec:b}

This section develops bounds on $|s_n(z) - \bar s_n(z)|$ and $|\frac{d}{dz}(s_n(z) - \bar s_n(z))|$ based on the work of Zhidong Bai, Jack Silverstein, Jiang Hu, and Wang Zhou, in particular, Section 4.1 of \cite{BHZ12} and Section 8 of \cite{BS_SpAn}.

\begin{theorem} \label{thrm:s-s0} Let $\eta > 0$ and $u_n \geq 1+(2+\eta)\sqrt{\beta_n}$ be a bounded sequence. For any $\ell < 1/4$, \mbox{almost surely}, 
\begin{align} \phantom{\,.} \sup_{|z - u_n| \leq n^{-1/4}\sqrt{\beta_n}} \, |s_n(z) - \bar s_n(z) | \lesssim   n^{-\ell} \beta_n^{-1/2}  \,. \end{align}
\end{theorem} 
\begin{corollary}  \label{cor:s'-s0'}
Let $\eta > 0$ and $u_n \geq 1 + (2+\eta)\sqrt{\beta_n}$ be a bounded sequence. For any $\ell < 1/8$, \mbox{almost surely},
\begin{align} \phantom{\,.} \sup_{ |z-u_n| \leq n^{-1/8}\sqrt{\beta_n}} \, \Big| \frac{d}{dz} \big(s_n(z) - \bar s_n(z) \big) \Big| \lesssim n^{-\ell}   \beta_n^{-1}         \,.
\end{align}
\end{corollary}

 \noindent These bounds are novel in that (1) dependence on $\beta_n$ is more carefully tracked and (2) the bound applies to $z \geq 1 + (2+\eta) \sqrt{\beta_n}$  (rather than to $z$ with $\Im(z) > 1/\sqrt{m}$). (2) is critical since the proof of Theorem \ref{thrm:cosine} requires bounding $|s_n(z) - \bar s_n(z)|$ and $|\frac{d}{dz} (s_n(z) - \bar s_n(z))|$ along a contour centered on the real line. Furthermore, (2) above enables improvement of previous bounds using the concentration of the maximum eigenvalue of $S_n$. For example,
in place of deterministic bounds such as 
 \[ \phantom{\, ,} \frac{1}{n} \, \tr (S_n - z I_n)^{-1} (S_n - \bar z I_n)^{-1} \leq \Im(z)^{-2} \, , 
 \]
we will argue using Lemma \ref{lem:max_eig2} that $\frac{1}{n} \tr (S_n - z I_n)^{-1} (S_n - \bar z I_n)^{-1} \lesssim \beta_n^{-1} $ with high probability. 
As set forth in Section \ref{sec:2.2}, the entries of $X_n$ are assumed truncated and normalized as in Lemma \ref{lem:max_eig2}. We note that Theorem \ref{thrm:s-s0} and Corollary \ref{cor:s'-s0'} hold for  non-truncated noise as well; the effects of truncation and normalization may be shown to be negligible using Lemma \ref{lem:trunc}.

Let $x_k^\top$ denote the $k$th-row of $X_n$, $X_{nk}$ consist of the remaining $n-1$ rows, and $\E_k(\cdot)$ denote the conditional expectation given $\{x_{ij}, i \leq k, j \leq m\}$. Additionally, we define
\begingroup
\allowdisplaybreaks
\begin{align} 
\phantom{\,,}
s_{kk} &= \frac{1}{m} \sum_{j=1}^m |x_{kj}|^2 \, , \nonumber \\ S_{nk} &= \frac{1}{m} X_{nk} X_{nk}^\top                         \nonumber \,, \\ \phantom{\,,}
\alpha_k &= X_{nk} x_k \, ,  \nonumber\\ 
\sigma_k & =  \tr (S_n - z I_n)^{-1} - \tr(S_{nk} - z I_{n-1})^{-1}        \nonumber  \,,\\
\phantom{\,,}\gamma_k  &= (\E_k - \E_{k-1}) \sigma_k    \, ,   \nonumber \\
b_k  &= \frac{1}{1-z-\beta_n-\frac{z}{m} \tr (S_{nk} - z I_{n-1})^{-1}}                   \nonumber\,, \\ \phantom{\,,} 
\tilde b_k  &= \frac{1}{ s_{kk} - z - \frac{1}{m^2}\alpha_k^\top(S_{nk} - z I_{n-1})^{-1} \alpha_k}                   \, ,  \nonumber \\
 \bar b_n &= \frac{1}{1-z-\beta_n  - \beta_n z \E s_n(z)} \,, \nonumber \\ 
 \phantom{\,,} 
\epsilon_k & = s_{kk} - 1 - \frac{1}{m^2} \alpha_k^\top (S_{nk}  - zI_{n-1})^{-1} \alpha_k + \frac{1}{m} \tr (S_{nk} - z I_{n-1})^{-1} S_{nk} + \frac{1}{m} \,, \nonumber \\
\phantom{\,,} 
\bar \epsilon_k &= s_{kk} - 1 + \beta_n + \beta_n z \E s_n(z) - \frac{1}{m^2} \alpha_k^\top (S_{nk} - z I_{n-1})^{-1} \alpha_k  \,. \label{0010}
\end{align}
\endgroup
Standard properties of these definitions are stated below (see Theorem A.5, Lemma 8.18, and (8.4.19) of \cite{BS_SpAn}). 
\begin{lemma}\label{lem:5}
For $z \in \mathbb{C}^+$, 
\begin{align*}
\sigma_k &=   \frac{1 + \frac{1}{m^2} \alpha_k^\top (S_{nk} - z I_{n-1})^{-2} \alpha_k}{s_{kk} - z - \frac{1}{m^2}\alpha_k^\top (S_{nk}  - zI_{n-1})^{-1} \alpha_k}   \, ,   & |\sigma_k| &\leq \frac{1}{\Im(z)} \, ,  \\ |\tilde b_k| &\leq \frac{1}{\Im(z)} \, , & |\bar b_n| &\leq \frac{1}{\sqrt{  \beta_n |z| }} \, .
 \end{align*} 
Furthermore, for $\Im(z) \geq \frac{1}{\sqrt{m}}$, $|b_k| \lesssim (\beta_n |z|)^{-1/2}$.
\end{lemma}

In the following lemmas, we will use $u_n$ and $v_n$ to refer to the real and imaginary parts of a complex sequence $z_n$. The subscripts  will be suppressed for notational lightness. Lemmas \ref{lem:s-Es} and \ref{lem:Es-s0} respectively bound $|s_n(z) -  \E \bar s_n(z)|$ and $|\E s_n(z) - \bar s_n(z)|$ for $z$ with imaginary part at least $1/\sqrt{m}$. The proof of Theorem \ref{thrm:s-s0} extends the bound to the real axis.   
\begin{lemma} \label{lem:s-Es}
Let $z = u + i v$ be a bounded sequence where $u \geq 1 + (2+\eta)\sqrt{\beta_n}$ and  $v \geq 1/\sqrt{m}$. Then,
\[ \phantom{\, .}  \E |s_n(z) - \E s_n(z) |^2 \lesssim \frac{1}{n^2} \Big( \frac{1}{\beta_n} + \frac{1}{ v^2}  \Big) \,.
\]
\end{lemma}

\begin{proof} Begin with a standard decomposition of $s_n(z) - \E s_n(z)$: 
\begin{align*} s_n(z) - \E s_n(z) & = \frac{1}{n}\sum_{k=1}^n  (\E_k - \E_{k-1})\big( \tr(S_n - z I_n)^{-1} - \tr(S_{nk} - z I_{n-1})^{-1} \big)  \\
& =   \frac{1}{n}\sum_{k=1}^n \gamma_k  \, .
\end{align*}
As $\{\gamma_k\}_{k=1}^n$ forms a martingale difference sequence, the Burkholder inequality (Lemma 2.12 of \cite{BS_SpAn}) yields
\[ \phantom{\,.}   \E |s_n(z) - \E s_n(z) |^2  \lesssim \frac{1}{n^2} \sum_{k=1}^n \E |  \gamma_k |^2       \,. \]
This martingale decomposition, together with the bound $|\gamma_k| \leq 2/v$, appear in proofs of convergence to the Marchenko-Pastur distribution as $\beta_n \rightarrow \beta > 0$ \cite{A11,BZ08}. Such proofs consider $z$ fixed. Here and in \cite{BHZ12, BS_SpAn}, as $v$ may decay rapidly, a tighter bound on the moments of $\gamma_k$ is needed.  Using the identity
\begin{align} 
 \phantom{\,.} S_{nk} (S_{nk} - zI_{n-1})^{-1} = (S_{nk} - zI_{n-1})^{-1} S_{nk}  = I_{n-1} + z (S_{nk} - zI_{n-1})^{-1} \label{S(S-zI)^-1} \,,
\end{align} we obtain 
\begin{align*} & \phantom{\,,} \tilde b_k^{-1} - b_k^{-1} = \epsilon_k \, , & \tilde b_k  = \frac{b_k}{1 + b_k\epsilon_k} \,.
\end{align*}
Thus,
\begin{align*} \sigma_k &= \frac{b_k}{1 + b_k\epsilon_k}\Big(1 + \frac{1}{m^2} \alpha_k^\top(S_{nk} - z I_{n-1})^{-2} \alpha_k \Big)    = -\sigma_k b_k  \epsilon_k + b_k \Big(1 + \frac{1}{m^2} \alpha_k^\top (S_{nk} - z I_{n-1})^{-2} \alpha_k \Big)  
\end{align*}
and
\begin{align} \phantom{\,,} \gamma_k &= - (\E_k - \E_{k-1})  \sigma_k b_k\epsilon_k +  \frac{1}{m^2} (\E_k - \E_{k-1}) b_k \alpha_k^\top (S_{nk} - zI_{n-1})^{-2} \alpha_k  \nonumber  \\
&= - (\E_k - \E_{k-1}) \sigma_k b_k \epsilon_k +  \frac{1}{m^2} \E_k  b_k \big[ \alpha_k^\top (S_{nk} - zI_{n-1})^{-2} \alpha_k - \tr(X_{nk}^\top (S_{nk} - zI_{n-1})^{-2} X_{nk} ) \big]    \,.     \label{gamma_k}
\end{align}
Here, we have used that  $b_{k}$ does not depend on $x_k$, so $(\E_k - \E_{k-1}) b_k = 0$. By Lemmas \ref{lem:5} and \ref{lem:xAx-tr}, 
\begin{align}  & \frac{1}{m^4} \E b_k^2 \big| \alpha_k^\top (S_{nk}  -\,    z I_{n-1}  )^{-2} \alpha_k  -   \, \tr(X_{nk}^\top (S_{nk} - zI_{n-1})^{-2} X_{nk} )  \big|^2
  \\ \lesssim  & \, \frac{1}{m^4 \beta_n} \, \E   \tr \, X_{nk}^\top (S_{nk} - zI_{n-1})^{-2} X_{nk}  (X_{nk}^\top (S_{nk} - \bar z I_{n-1})^{-2} X_{nk} )^\top  \nonumber \\
   = &  \,   \frac{1}{n m } \E \tr \,  S_{nk} (S_{nk} - zI_{n-1})^{-2} S_{nk} (S_{nk} - \bar zI_{n-1})^{-2}   =  \frac{1}{ m } \E \bigg( \frac{1}{n} \sum_{i=1}^{n-1}   \frac{\lambda_{k,i}^2}{|\lambda_{k,i} - z|^4}  \bigg) \,, \label{957}
\end{align}
where $\lambda_{k,1} \geq \cdots \geq \lambda_{k,n-1}$ denote the eigenvalues of $S_{nk}$. By (\ref{111}) and a union bound, 
\[  \phantom{\,,} \Pr \big(\lambda_{k,1} \geq 1 + (2+\eta/2)\sqrt{\beta_n} \text{ for some } k \big) \lesssim \frac{1}{n} \,.
\]
On the complement of this event, the  integrand of (\ref{957}) is bounded by $16|z|/(\eta \sqrt{\beta_n})^4$. In addition, 
\begin{align*} &  \frac{\lambda_{k,i}^2}{|\lambda_{k,i} - z|^2} \leq 1 + \frac{u^2}{v^2} \, , &   \frac{\lambda_{k,i}^2}{|\lambda_{k,i} - z|^4}  \leq \frac{|z|^2}{v^4} \,.
\end{align*}
Therefore,
\begin{align}
 \phantom{\,.} \frac{1}{m^4} \E b_k^2 \big| \alpha_k^\top (S_{nk} - zI_{n-1})^{-2} \alpha_k -  \tr(X_{nk}^\top (S_{nk} - zI_{n-1})^{-2} X_{nk} )  \big|^2 \lesssim \frac{1}{m} \Big(\frac{1}{\beta_n^2} + \frac{1}{n v^4} \Big) \,.
\end{align}
Note that \cite{BS_SpAn} bounds terms such as $(\ref{957})$ deterministically.
Similarly,
\begin{align} \phantom{\,.}  \E | \epsilon_k|^2 & \lesssim \E  |s_{kk} - 1|^2  + \frac{1}{m^2}  +  \frac{1}{m^4} \E |\alpha_k^\top (S_{nk} - zI_{n-1})^{-1} \alpha_k -  \tr(X_{nk}^\top (S_{nk} - zI_{n-1})^{-1} X_{nk} )  |^2 \nonumber \\
& \lesssim \frac{1}{m}  +  \frac{1}{m^2} \E \tr \, S_{nk} (S_{nk} - zI_{n-1})^{-1} S_{nk} (S_{nk} - \bar{z} I_{n-1})^{-1}       \lesssim  \frac{1}{m} \Big(1 + \frac{1}{m v^2} \Big) \label{eps_bound} \,.
\end{align} 
By Lemma \ref{lem:5} and (\ref{eps_bound}),
\begin{align}
   \phantom{\,.}  \E  |\sigma_k b_k \epsilon_k|^2 & \lesssim \frac{1}{v^2 \beta_n}  \E | \epsilon_k|^2 \lesssim \frac{1}{n v^2}\,. \label{faqqot}
\end{align}
Thus, using (\ref{gamma_k}) - (\ref{faqqot}) and $v \geq 1 /\sqrt{m}$,
\begin{align}\phantom{\,.}  \E |s_n(z) - \E s_n(z) |^2 & \lesssim \frac{1}{n^2} \sum_{k=1}^n  \E |  \gamma_k |^2   \nonumber  \lesssim \frac{1}{n^2} \Big( \frac{1}{\beta_n} + \frac{1}{ v^2} \Big)  \,.
\end{align}
\end{proof}

\begin{lemma} \label{lem:Es-s0}
Let $z = u + i v$ be a bounded sequence where $u \geq 1 + (2+\eta)\sqrt{\beta_n}$ and  $v \geq 1/\sqrt{m}$. Then,
\[ \phantom{\, .}  | \E s_n(z) - \bar s_n(z) | \lesssim \frac{1}{nv} \,.
\]
\end{lemma}

\begin{proof}
Recall definitions (\ref{0010}). We have
\begin{align*} & \phantom{\,,} \tilde b_k^{-1} - \bar b_n^{-1} = \bar \epsilon_k\, , & \tilde b_k - \bar b_n  = -\bar b_n^2 \bar \epsilon_k + \bar b_n^2 \tilde b_k \bar \epsilon_k^2 \,.
\end{align*}
Together with Theorem A.4 of \cite{BS_SpAn}, this leads to the following decomposition of $s_n(z)$: 
\begin{align}\phantom{\,.} s_n(z) &= \frac{1}{n}\sum_{k=1}^n \tilde b_k =  \bar b_n + \frac{1}{n} \sum_{k=1}^n (   -\bar b_n^2 \bar \epsilon_k + \bar b_n^2 \tilde b_k \bar \epsilon_k^2 ) \label{234} \,.
\end{align}
(8.3.12) of \cite{BS_SpAn} proves the moment bound $| \E \bar \epsilon_k |  \leq 1/(mv)$. The second moment of $\bar \epsilon_k$ is bounded as follows:  
\begin{align} \phantom{\,.} \E | \bar \epsilon_k|^2  &  \lesssim  \E | \bar \epsilon_k - \E(\bar \epsilon_k | X_{nk}) |^2 + \E | \E ( \bar \epsilon_k | X_{nk} ) - \E \bar \epsilon_k |^2 + |  \E \bar \epsilon_k |^2   \, ,
\end{align}
A bound on $\E | \bar \epsilon_k - \E(\bar \epsilon_k | X_{nk}) |^2 =  \E | \epsilon_k - m^{-1} |^2 \leq 2(\E |\epsilon_k|^2 + m^{-2})$ is given by (\ref{eps_bound}). By (\ref{S(S-zI)^-1}) and Lemma \ref{lem:5},
\begin{align}\nonumber
    \E | \E ( \bar \epsilon_k | X_{nk} ) - \E \bar \epsilon_k |^2 & =  \frac{1}{m^2} \E \big| \tr (S_{nk} - zI_{n-1})^{-1} S_{nk} - \E \tr (S_{nk} - zI_{n-1})^{-1} S_{nk} \big|^2 \\ & = \frac{|z|^2}{m^2} \E \big| \tr (S_{nk} - zI_{n-1})^{-1} - \E \tr (S_{nk} - zI_{n-1})^{-1} \big|^2 \nonumber \\
    & \leq |z|^2 \beta_n^2 \E | s_n(z) - \E s_n(z)|^2 + \frac{2|z|^2}{m^2 v^2} \,.   \label{994}
\end{align}
By (\ref{234}) -  (\ref{994}) and Lemma \ref{lem:s-Es},
\begin{align} \nonumber
    \phantom{\, .} | \E s_n(z) - \bar b_n| & \lesssim \frac{1}{n \beta_n} \sum_{k=1}^n | \E \bar \epsilon_k|  + \frac{1}{n v \beta_n} \sum_{k=1}^n \E |\bar \epsilon_k|^2  \\ & \lesssim \frac{1}{n v} + \frac{1}{v \beta_n}  \Big( \frac{1}{m} +\frac{1}{m^2 v^2}  + \beta_n^2 \E |s_n(z) - \E s_n(z) |^2\Big)    \lesssim \frac{1}{n v} \,.
    \label{998}
\end{align}
The above bound describes an equation that is quadratic in $\E s_n(z)$:
\begin{align} \phantom{\,,} \E s_n(z) -  \frac{1}{1-z-\beta_n  - \beta_n z \E s_n(z)} = \delta_n \, , \label{0100}
\end{align}
where $\delta_n \coloneqq \E s_n(z) -\bar b_n$ obeys $|\delta_n|  \lesssim 1/(nv) \leq 1 /\sqrt{n \beta_n}$ (note that $\bar s_n(z)$ is a solution of the related equation $s = (1-z-\beta_n z s)^{-1})$. 
By (3.3.8)-(3.3.13) of \cite{BS_SpAn}, the unique solution satisfying the requirement that $ \Im(\E s_n(z)) > 0$ (recall that $z \in \mathbb{C}^+$) is given by
\[ \phantom{\,.} \E s_n(z) = \frac{1}{2 \beta_n z} \Big (1 - z - \beta_n + \beta_n z \delta_n + \sqrt{(z + \beta_n -1 + \beta_n z \delta_n)^2 - 4 \beta_n z} \,\Big) \,.
\]
Thus, 
\begin{align} \phantom{\,.} |\E s_n(z) - \bar s_n(z)| & =  \frac{1}{2} \bigg| \delta_n + \frac{2  (z+\beta_n-1)\delta_n  + \beta_n z \delta_n^2}{\sqrt{(z+\beta_n-1)^2-4\beta_n z} + \sqrt{(z+\beta_n-1 +\beta_n z \delta_n)^2-4\beta_n z} } \bigg| \nonumber \\ 
& \leq \frac{|\delta_n|}{2} \bigg( 1 + \frac{|2(z+\beta_n-1) + \beta_n z \delta_n|}{|\sqrt{(z+\beta_n-1)^2-4\beta_n z} + \sqrt{(z+\beta_n-1 + \beta_n z \delta_n)^2-4\beta_n z}|} \bigg) \lesssim \frac{1}{nv} \label{993} \,.
\end{align}
\end{proof}

\begin{proof} [Proof of Theorem \ref{thrm:s-s0}]
Let $z$ satisfy $|z-u| \leq n^{-1/4} \sqrt{\beta_n}$ and $z' = u +  i n^{-1/4}\sqrt{\beta_n}$. We will decompose $|s_n(z) - \bar s_n(z)|$ as
\begin{align*}  |s_n(z) - \bar s_n(z)|  & \leq |s_n(z) - s_n(z') | 
 + | s_n(z')  - \bar s_n(z') |   +  |  \bar s_n(z)   - \bar s_n(z') | 
\end{align*}
and bound each term. The first is bounded in the proof of Lemma \ref{lem123}. By Lemmas 
\ref{lem:s-Es} and \ref{lem:Es-s0}, 
\begin{align} \phantom{\,.} | s_n(z') - \bar s_n(z') | \lesssim  \frac{1}{n^{\ell} \sqrt{\beta_n}} \,, \label{992} 
\end{align}
almost surely. As $u$ is bounded away from the support of $\overbar F_n$ by a constant multiple of $\sqrt{\beta_n}$ (\ref{1000}),
\begin{align} \phantom{\,.} | \bar s_n(z) - \bar s_n(z')| & \leq \int   \Big| \frac{1}{\lambda -z} - \frac{1}{\lambda - z'} \Big| d \overbar F_n(\lambda)  \nonumber = \int    \frac{|z-z'|}{|\lambda - z||\lambda - z'|} d \overbar F_n(\lambda) \nonumber \\
& \lesssim \frac{1}{n^{1/4} \sqrt{ \beta_n}} \, . \label{996}
\end{align} 
\end{proof}

\begin{proof}[Proof of Corollary \ref{cor:s'-s0'}]
Consider $z$ satisfying $|z - u| \leq n^{-1/8} \sqrt{\beta_n}$. Let $v = n^{-1/8} \sqrt{\beta_n}$ and $z' = u +  i v$. We will decompose $|\frac{d}{dz} (s_n(z) - \bar s_n(z) ) |$ as
\begin{align}
 \Big|\frac{d}{dz} \big(s_n(z) - \bar s_n(z) \big) \Big| & \leq \Big|\frac{d}{dz} \big(s_n(z) -  s_n(z') \big) \Big| + \Big|\frac{d}{dz} \big(s_n(z') - \bar s_n(z') \big) \Big|    \nonumber  + \Big|\frac{d}{dz} \big(\bar s_n(z) - \bar s_n(z') \big) \Big| \nonumber
\end{align}
and bound each term. For $1 \leq \alpha \leq n$, by the spacing bound (\ref{419}), 
\[ \phantom{\,.} \bigg| \frac{1}{(\lambda_\alpha - z)^2} - \frac{1}{(\lambda_\alpha - z')^2} \bigg| = \frac{|z-z'||2\lambda_\alpha -z-z'|}{|\lambda_\alpha-z|^2 |\lambda_\alpha - z'|^2} \leq \frac{32}{\eta^3 n^{1/8} \beta_n} \,.
\]
Thus, almost surely,
\begin{align}
     \Big|\frac{d}{dz} \big(s_n(z) -  s_n(z') \big) \Big| & \leq \frac{1}{n}  \sum_{\alpha = 1}^n \bigg| \frac{1}{(\lambda_\alpha - z)^2} - \frac{1}{(\lambda_\alpha - z')^2} \bigg|  \lesssim \frac{1}{n^{1/8} \beta_n} \, . 
\end{align}
Note that 
\begin{align*}    &  \bigg|\frac{1}{(\lambda - z')^2} - \frac{1}{(\lambda - u)^2+v^2} \bigg| \leq \frac{2v}{ |\lambda - u|^3}  \,, &  \frac{1}{(\lambda - u)^2 + v^2} = \Im \Big( \frac{1}{v(\lambda - z)} \Big) \,.
\end{align*}
Hence, 
 \begin{align*} \Big|\frac{d}{dz} \big(s_n(z') - \bar s_n(z') \big) \Big| &= \bigg| \int \frac{1}{(\lambda - z')^2} d F_n(\lambda) -  \int \frac{1}{(\lambda - z')^2} d \overbar F_n(\lambda)  \bigg| \\
& \leq \bigg| \int \frac{1}{(\lambda - z')^2} - \frac{1}{(\lambda - u)^2 + v^2} (dF_n - d \overbar F_n)(\lambda) \bigg| + \frac{1}{v} \bigg| \int \frac{1}{\lambda - z'} (dF_n - d \overbar F_n)(\lambda)  \bigg| \\
& \leq 2 v \int  \frac{1}{|\lambda - u|^3} (dF_n + d \overbar F_n)(\lambda) + \frac{1}{v} | s_n(z') 
- \bar s_n(z')| \, . 
\end{align*}
We bound the first term above as in Lemma \ref{lem:s_n bound} and the second using Theorem \ref{thrm:s-s0}: almost surely,
\begin{align}  \phantom{\,.} \Big|\frac{d}{dz} \big(s_n(z') - \bar s_n(z') \big) \Big| \lesssim \frac{v}{ \beta_n^{3/2}} + \frac{1}{n^{\ell+1/8} v \sqrt{\beta_n}} \,. 
\end{align} 
Finally, by (\ref{1000}),
\begin{align}
    \Big|\frac{d}{dz} \big(\bar s_n(z) - \bar s_n(z') \big) \Big| & \leq \int \bigg|\frac{1}{(\lambda - z)^2} - \frac{1}{(\lambda - z')^2} \bigg|  d \overbar F_n(\lambda)  \lesssim \frac{1}{n^{1/8} \beta_n} \, . 
\end{align}
\end{proof}

\subsection{Acknowledgements}

The author is grateful to David Donoho for continuous encouragement and detailed comments on drafts of this paper.  The author would also like to thank the anonymous referees for their insightful suggestions and comments. This work was partially supported by NSF DMS grants 1407813, 1418362, and 1811614.

\appendix
\section{Appendix}

\begin{lemma} \label{lem:max_eig} (Theorem 1 of \cite{CP12}) 
Let $X_n = (x_{ij}: 1 \leq i \leq n, 1 \leq j \leq m)$ be an array of i.i.d.\ real-valued random variables with $\E x_{11} = 0$, $\E x_{11}^2 = 1$, and $\E x_{11}^4 < \infty$. Suppose that as $n \rightarrow \infty$, $\beta_n = n/m_n \rightarrow 0$. Define
\[\phantom{\,.} A_n = \frac{1}{2 \sqrt{\beta_n}} \Big(\frac{1}{m}X_n X_n^\top -  I_n \Big) \, . \]
Then,
\[
\phantom{\,,} \lambda_{max}(A_n) \xrightarrow{a.s.}  1\,,
\]
where $\lambda_{max}(A_n)$ represents the largest eigenvalue of $A_n$.
\end{lemma}

\begin{lemma} \label{lem:max_eig2} (Theorem 2 of \cite{CP12})
In the setting of Lemma \ref{lem:max_eig}, let  $\widehat X_n =(\hat x_{ij}: 1 \leq i \leq n, 1 \leq j \leq m)$ and  $Y_n  =(y_{ij}: 1 \leq i \leq n, 1 \leq j \leq m) $ denote $X_n$ after truncation and normalization, respectively:
\begin{align*} \phantom{\,, } & \hat x_{ij} = x_{ij} I(|x_{ij}| \leq \delta_n (n m)^{1/4} ) \,, && y_{ij} = \frac{\hat x_{ij} - \E \hat x_{11}}{\nu}  \,, \end{align*}
where $\nu^2 = \E (\hat x_{11} - \E \hat x_{11})^2$ and $\delta_n$ is a sequence constructed in Section 2 of \cite{CP12}, satisfying $\delta_n \rightarrow 0$ and $\delta_n (nm)^{1/4} \rightarrow \infty$. Define
\[\phantom{\,.} \widetilde A_n = \frac{1}{2 \sqrt{\beta_n}} \Big(\frac{1}{m} Y_n Y_n^\top -  I_n \Big) \, . \]
Then, for any $\eta, \ell > 0$,
\[ \phantom{\,.} \Pr(\lambda_{max}(\widetilde A_n) \geq 1 + \eta) = o(n^{-\ell}) \, . 
\]
Furthermore, $\Pr(X_n \neq \widehat X_n \,\, i.o.) = 0$, 
\begin{align*}
  &  | \E \hat x_{11} | \lesssim \frac{1}{(n m )^{3/4}} \,, && |\nu^2 - 1| = o \Big(\frac{1}{\sqrt{n m}} \Big) \, .  
\end{align*}
\end{lemma}

\begin{lemma}\label{lem:trunc}
In the setting of Lemma \ref{lem:max_eig2}, let $P_n \in \mathbb{R}^{n \times m}$ have bounded operator norm and consider the singular value decompositions
\begin{align*}
    & P_n + \frac{1}{\sqrt{m}} X_n = U \Lambda V^\top \, , &&  P_n + \frac{1}{\sqrt{m}} Y_n = \widetilde U \widetilde \Lambda \widetilde V^\top \, ,
\end{align*}
where $\Lambda = \mathrm{diag}(\sigma_1, \ldots, \sigma_n)$ and  $\widetilde \Lambda = \mathrm{diag}(\tilde \sigma_1, \ldots, \tilde \sigma_n)$.  Almost surely, 
\[ \phantom{\,,} \sup_{1 \leq i \leq n} | \sigma_i^2 - \tilde \sigma_i^2| \lesssim \frac{1}{\sqrt{n m}} \,.
\]
Moreover,  for fixed $i \leq n$, if $\min(\sigma_{i}^2 - \sigma_{i-1}^2, \sigma_i^2 - \sigma_{i+1}^2) \asymp \sqrt{\beta_n}$, where $\sigma_0 \coloneqq \infty$ and $\sigma_n \coloneqq -\infty$, then
\begin{align*}
    & 1 - | (U^\top \widetilde U)_{ii} | \lesssim \frac{1}{n^2} \, , & 1 - |(V^\top \widetilde V)_{ii} | \lesssim \frac{1}{n^2} \,.
\end{align*}
\end{lemma}
\begin{proof}
By Lemmas \ref{lem:max_eig} and \ref{lem:max_eig2}, almost surely eventually, $\|X_n\|_2 \leq 2 \sqrt{m}$, $\|Y_n\|_2 \leq 2 \sqrt{m}$, and
\begin{align} \phantom{\,.} \frac{1}{\sqrt{m}}\|  X_n - Y_n \|_2 & =  \frac{1}{\sqrt{m}}\| \widehat X_n - Y_n \|_2  = \frac{1}{\sqrt{m}} \|\widehat X_n - \nu Y_n - (1-\nu)Y_n \|_2 \nonumber \\ 
&\leq \frac{1}{\sqrt{m}} \| (\E \hat x_{11} )  {\bf 1}_n {\bf 1}_m^\top \|_2 + \frac{\nu-1}{\sqrt{m}} \|Y_n\|_2 
 \lesssim \frac{1}{\sqrt{n m }} \,. \label{6489}
\end{align}
The first claim of the lemma follows from (\ref{6489}) and Weyl's inequality: 
\begin{align} \phantom{\,.} |\sigma_i^2 -  \tilde \sigma_i^2|  & =  | \sigma_i -  \tilde \sigma_i|   |\sigma_i + \tilde \sigma_i | \leq   \frac{1}{\sqrt{m}} \| X_n - Y_n \|_2 |\sigma_1 + \tilde \sigma_1| \,. \nonumber
\end{align}
The second claim follows from the Davis-Kahan theorem (Corollary 3 of \cite{Davis}).
\end{proof}

\begin{lemma}  \label{lem:xAx-tr}  (Lemma 2.7 of \cite{BS98}) Let $A$ be an $n \times n$ nonrandom matrix and $x = (x_1, \ldots, x_n)^\top$ be a random vector of independent entries. Assume that $\E x_i = 0$, $\E |x_i|^2 = 1$, and $\E |x_j|^\ell \leq \nu_\ell$. Then, for any $\ell \geq 1$,
\[ \phantom{\,,} \E | x^* A x - \tr A |^\ell \lesssim_\ell  (\nu_4 \tr A^* A )^{\ell/2} + \nu_{2\ell} \tr(A^* A)^{\ell/2}  \,.
\] 
\end{lemma}

\begin{lemma} \label{lem:xAy} Let $A$ be an $n \times m$ nonrandom matrix and $x \in \mathbb{C}^n$, $y \in \mathbb{C}^m$ random vectors. Assume $x$ and $y$ are independent with entries satisfying the moment conditions of Lemma \ref{lem:xAx-tr}.  Then, for any $\ell \geq 2$, 
\[ \phantom{\,,}\E | x^* A y|^{\ell} \lesssim_\ell  (1+ \nu_4^{\ell/4} + \nu_\ell) \big( (\nu_4 \tr(A^* A)^2 )^{\ell/4} + \nu_\ell \tr(A^* A)^{\ell/2} + ( \tr A^* A )^{\ell/2} \big) \,.
\]
\end{lemma}
\begin{proof}
Condition on $y$ and apply Lemma \ref{lem:xAx-tr}.
\end{proof}

Lemmas \ref{lem:detA-B} and \ref{lem:d/dz detA-detB} below are elementary, following from the Leibniz determinant formula. 
\begin{lemma} \label{lem:detA-B}
For any $n \times n$ matrices $A$ and $B$,
\[ \phantom{\,.} | \det A - \det B | \lesssim_n (\|A\|_\infty + \|B \|_\infty)^{n-1} \|A - B \|_\infty \,.
\]
\end{lemma}

\begin{lemma} \label{lem:d/dz detA-detB}
For any $n \times n$ matrices $A(z)$ and $B(z)$,
\[ \phantom{\,.} \Big| \frac{d}{dz}  (\det A - \det B ) \Big| \lesssim_n \|A\|_\infty^{n-1} \Big\| \frac{d}{dz} (A-B) \Big\|_\infty + (\|A\|_\infty + \|B\|_\infty)^{n-2} \Big \| \frac{d}{dz} B \Big \|_\infty\|A - B \|_\infty \,,
\]
where $\frac{d}{dz}B$ denotes the entrywise derivative of the matrix.
Additionally,
\[ \phantom{\,,} \Big| \frac{d}{dz} \det A \Big| \lesssim_n \|A\|_\infty^{n-1} \Big\| \frac{d}{dz} A \Big \|_\infty \, . 
\]
\end{lemma}

 The remaining lemmas pertain to the right singular vectors of the noise matrix. Consider an array $X_n$ satisfying assumption A1 and a vector $v$ satisfying the assumptions A2 places on right signal vectors. Let $W$ be a matrix containing as columns the first $n$ (normalized) eigenvectors of $X_n^\top X_n$. 

\begin{lemma} \label{lem:6.8} Let $S_n = \frac{1}{m} X_n X_n^\top$ and $A_n = \frac{1}{m} X_n^\top S_n^{-1} X_n = W W^\top$. Almost surely, 
\[ \phantom{\,.} v^\top A_n v = \|W^\top v\|_2^2  \lesssim \beta_n \log(n)
\,. \]
\end{lemma}
\noindent The proof requires the following two lemmas:  

\begin{lemma} \label{lem:6.9}
 Let $n/m \leq \beta$ for some $\beta \in (0, 1]$. There exists constants $c_0, c_1, c_2$ depending only on $\E x_{11}^4$ such that with probability at least $1 - c_0 \log(e/\beta) \exp(-c_1 \beta m)$,
\[ \phantom{\,.}   \frac{1}{\sqrt{m}} \sigma_{\min}( X_n ) \geq 1 - c_2 \sqrt{\beta} \,.
\]
\end{lemma}
\begin{proof}
This is one case of Theorem 1.3 of \cite{KM15}.
\end{proof}

\begin{lemma} \label{lem:6.10}
Let $\widetilde A_n = \frac{1}{m} X_n^\top (S_n + \frac{1}{m} I_n)^{-1} X_n$. The fourth moment condition on the entries of $\sqrt{m} v$ implies
\begin{align*} & \E (v^\top \widetilde A_n v) \leq \beta_n \, , & Var(v^\top \widetilde A_n v) \lesssim \frac{\beta_n^2}{n}  \,. 
\end{align*}
\end{lemma}

\begin{proof}[Proof of Lemma \ref{lem:6.8}]Let $SD$ denote standard deviation. By Lemma \ref{lem:6.10}, eventually,
\begin{align*} \phantom{\,.} \text{Pr} ( v^\top \widetilde A_n v \geq  \beta_n \log(n) ) & \leq \text{Pr} \big(v^\top \widetilde A_n v \geq \E(v^\top \widetilde A_n v) + \sqrt{n} \log(n) \cdot SD(v^\top \widetilde A_n v)\big)  \leq \frac{1}{n \log^2(n)} \, . 
 \end{align*}
Summability of the right-hand-side gives $v^\top \widetilde A_n v \leq \beta_n \log(n)$ almost surely eventually.

Let $\lambda_1 \geq \cdots \geq \lambda_n$ denote the eigenvalues of $S_n$ and $\Lambda = \text{diag}(\lambda_1, \ldots, \lambda_n)$. Observe that the eigenvalues of $\widetilde A_n = W \Lambda^{1/2} ( \Lambda + \frac{1}{m}I_n)^{-1} \Lambda^{1/2} W^\top$ are $\lambda_j/(\lambda_j + \frac{1}{m})$, $j = 1, \ldots, n$, joined by a zero eigenvalue with multiplicity $m - n$. A first consequence of Lemma \ref{lem:6.9} is $\lambda_n \gtrsim 1$ almost surely. Hence, 
\begin{align*} \phantom{\,.} v^\top A_n v  & \leq v^\top \widetilde A_n v + \|v\|_2^2 \lambda_{\max}(A_n - \widetilde A_n)  \leq v^\top \widetilde A_n v + \frac{\|v\|_2^2}{1 + m \lambda_n } \\
& \lesssim \beta_n \log(n) \, , 
\,\end{align*}
almost surely.
\end{proof}

\begin{proof}[Proof of Lemma \ref{lem:6.10}]
As $X_n$ and $v$ are independent, 
\begin{align}
    \E (v^\top \widetilde A_n v)^2 & = \sum_{j_1,j_2,j_3,j_4=1}^m \E(v_{j_1} v_{j_2} v_{j_3} v_{j_4}) \E ( \widetilde A_{j_1 j_2} \widetilde A_{j_3 j_4} ) \, .  \label{9876}
\end{align}
Any term $\E(v_{j_1} v_{j_2} v_{j_3} v_{j_4})$ containing a singleton index vanishes, leaving the following terms: 
\begin{align}
\begin{aligned}
j_1 = j_2 = j_3 = j_4  : \qquad & \frac{\nu_4}{m^2} \sum_{j}\E\widetilde A_{j j}^2 \, , \\
    j_1 =  j_2 , j_3 = j_4, j_1 \neq j_3 : \qquad & \frac{1}{m^2}  \sum_{j_1,j_3}  \E(\widetilde A_{j_1 j_1} \widetilde A_{j_3 j_3}) \,, \\
    j_1 = j_3 , j_2 = j_4, j_1 \neq j_2  : \qquad& \frac{1}{m^2} \sum_{j_1,j_2}   \E \widetilde A_{j_1 j_2}^2 \,, \\
    j_1 = j_4 , j_2 = j_3, j_1 \neq j_2 :\qquad & \frac{1}{m^2} \sum_{j_1,j_2}   \E\widetilde A_{j_1 j_2}^2 \,,
\end{aligned}
\end{align}
where $\nu_4 = m^2 \E v_1^4 < \infty$. Let $x_j$ denote the $j$-th column of $X_n$ and $B_j = S_n  +  \frac{1}{m} (I_n - x_j x_j^\top)$. By the Sherman-Morrison formula,
\begin{align*}
    \E  \widetilde A_{j j}^2 & = \frac{1}{m^2} \E \big( x_j^\top (S_n + \tfrac{1}{m} I_n)^{-1} x_j \big)^2 
    = \frac{1}{m^2} \E  \bigg( x_j^\top B_j^{-1} x_j  - \frac{(x_j^\top B_j^{-1} x_j  )^2}{m + x_j^\top B_j^{-1} x_j  } \bigg)^2 \\
    & = \frac{1}{m^2} \E \bigg( \frac{x_j^\top B_j^{-1} x_j}{1  + \tfrac{1}{m} x_j^\top B_j^{-1}x_j}  \bigg)^2 
     \leq \frac{1}{m^2} \E ( x_j^\top B_j^{-1} x_j )^2 \leq \frac{1}{m^2} \E\|x_j\|_2^4 \, \E \lambda_{\max}(B_j^{-1})^2 \, ,
\end{align*}
where the final inequality follows from the independence of $x_j$ and $B_j$. Lemma \ref{lem:6.9} and the almost-sure lower bound $\lambda_{\min}(B_j) \geq  1/m $ imply  $\E  \lambda_{\max}(B_j^{-1})^2 \lesssim 1$. Moreover, $\E \|x_j\|_2^4 = n \nu_4 + (n^2-n)\nu_2^2$. Thus,
\begin{align} \phantom{\,.}
     \E | \widetilde A_{j j}|^2  \lesssim \beta_n^2 \, . 
\end{align}
Next, consider terms of the third and fourth types. Let $B_{j_1,j_2} = S_n + \frac{1}{m}(I_n - x_{j_1} x_{j_1}^\top - x_{j_2} x_{j_2}^\top)$. By the Woodbury formula, 
\begin{align*}
    \E \widetilde A_{j_1,j_2}^2 &= \frac{1}{m^2} \E \big( x_{j_1}^\top (S_n + \tfrac{1}{m} I_n)^{-1} x_{j_2} \big)^2\\ 
    & = \frac{1}{m^2} \E \bigg(  \frac{x_{j_1}^\top B_{j_1,j_2}^{-1} x_{j_2}}{ (1 + \tfrac{1}{m}  x_{j_1}^\top B_{j_1,j_2}^{-1} x_{j_1}) ( 1 + \tfrac{1}{m}  x_{j_2}^\top B_{j_1,j_2}^{-1} x_{j_2}) - (\tfrac{1}{m}   x_{j_1}^\top B_{j_1,j_2}^{-1} x_{j_2})^2 } \bigg)^2  \\
    & \leq \frac{1}{m^2} \E (x_{j_1}^\top B_{j_1,j_2}^{-1} x_{j_2})^2 \leq \frac{1}{m^2} \E \langle x_{j_1}, x_{j_2} \rangle^2 \E \lambda_{\max}(B_{j_1,j_2}^{-1})^2 \,.
\end{align*}
The second-to-last inequality follows from the Cauchy-Schwarz inequality. The final step uses the independence of $B_{j_1,j_2}$ from  $x_{j_1}$ and $x_{j_2}$. We have that $\E \lambda_{\max}(B_{j_1,j_2}^{-1})^2 \lesssim 1$ and $ \E \langle x_{j_1}, x_{j_2} \rangle^2 = n$. Therefore, 
\begin{align} \phantom{\,.}  \E \widetilde A_{j_1,j_2}^2 \lesssim \frac{\beta_n}{m}                      \,.
\end{align}
As the non-zero eigenvalues of $\widetilde A_n = W \Lambda^{1/2} ( \Lambda + \frac{1}{m}I_n)^{-1} \Lambda^{1/2} W^\top$ are $\lambda_j/(\lambda_j + \frac{1}{m})$, $j \in \{1, \ldots, n\}$, 
\begingroup
\allowdisplaybreaks
\begin{align*}
    &\E ( v^\top \widetilde A_n v) = \frac{1}{m} \E \tr \widetilde A_n  = \frac{1}{m} \sum_{j=1}^n \E \Big( \frac{\lambda_j}{\lambda_j + \frac{1}{m}} \Big) \leq \beta_n \,, &   & \E | \tr \widetilde A_n - n |    \leq   \sum_{j=1}^n \E \Big| \frac{1}{1 + m \lambda_j} \Big|   \lesssim \beta_n \, , \\
    & \E \big| (\tr \widetilde A_n)^2 - n^2 \big|   \leq \sum_{j_1, j_2=1}^n  \E \Big| \frac{1 + m(\lambda_{j_1} + \lambda_{j_2} ) }{(1 + m \lambda_{j_1}) (1 + m \lambda_{j_2} ) }  \Big| \lesssim n \beta_n \,.
\end{align*}
\endgroup

We have used that by Lemma \ref{lem:6.9}, there exists $c > 0$ such that with high probability, $\lambda_n \geq c$. On this event, 
\begin{align*} \phantom{\,.}  \Big| \frac{1 + m(\lambda_{j_1} + \lambda_{j_2} ) }{(1 + m \lambda_{j_1}) (1 + m \lambda_{j_2} ) } \Big| & = \frac{ \lambda_{j_1}^{-1} \lambda_{j_2}^{-1} + m(\lambda_{j_1}^{-1} + \lambda_{j_2}^{-1})  }{\lambda_{j_1}^{-1} \lambda_{j_2}^{-1} + m(\lambda_{j_1}^{-1} + \lambda_{j_2}^{-1})  + m^2} \leq \frac{c^{-2} + 2c^{-1} m}{c^{-2} + 2c^{-1} m + m^2} \,. \end{align*}
Thus, 
\begin{align}
      \bigg| \frac{1}{m^2}  \sum_{j_1,j_3}  \E(\widetilde A_{j_1 j_1} \widetilde A_{j_3 j_3})  - ( \E v^\top \widetilde A_n v )^2 \bigg| &  = \frac{Var(\tr \widetilde A_n)}{m^2} \nonumber \\ & \leq  \frac{1}{m^2} | \E (\tr \widetilde A_n)^2 - n^2| +  \frac{1}{m^2} | ( \E \tr \widetilde A_n)^2 - n^2|  \lesssim \frac{\beta_n^2}{m} \, . \label{98765}
\end{align}
$Var(v^\top \widetilde A_n v) \lesssim n^{-1} \beta_n^2$ follows from (\ref{9876}) - (\ref{98765}). 
\end{proof}

\end{document}